\documentclass{article}

\usepackage{amsmath,amssymb,amsthm}
\usepackage[margin=1in]{geometry}
\usepackage{hyperref}
\usepackage{authblk}
\usepackage{xcolor}

\theoremstyle{plain}
\newtheorem{theorem}{Theorem}[section]
\newtheorem{proposition}[theorem]{Proposition}

\newtheorem{corollary}[theorem]{Corollary}
\theoremstyle{definition}
\newtheorem{definition}[theorem]{Definition}
\theoremstyle{remark}
\newtheorem{remark}[theorem]{Remark}
\newtheorem{example}[theorem]{Example}

\newcommand{\C}{\mathbb C}
\newcommand{\R}{\mathbb R}
\newcommand{\E}{\mathbb E}
\newcommand{\PP}{\mathbb P}
\newcommand{\U}{\mathcal U}
\newcommand{\Sph}{\mathbb S}
\newcommand{\tr}{\operatorname{tr}}
\newcommand{\Tr}{\operatorname{Tr}}
\newcommand{\Law}{\operatorname{Law}}
\newcommand{\diag}{\operatorname{diag}}
\newcommand{\Id}{\mathbf 1}
\newcommand{\dd}{\,\mathrm d}
\newcommand{\norm}[1]{\left\lVert #1\right\rVert}
\newcommand{\Var}{\operatorname{Var}}
\newcommand{\supp}{\operatorname{supp}}

\title{Rank-One HCIZ Integrals and Multiplicative Convolutions}
\author[1]{Beno\^{\i}t Collins}
\author[2]{Nicolas Delporte}
\author[1]{Manasa Nagatsu}
\author[2]{Reiko Toriumi}
\affil[1]{Department of Mathematics, Graduate School of Science,
Kyoto University, Kyoto 606-8502, Japan}
\affil[2]{Gravity, Quantum Geometry and Field Theory Unit,
Okinawa Institute of Science and Technology Graduate University,
Okinawa 904-0495, Japan}

\date{}

\begin{document}
\maketitle

\begin{abstract}
We study how spectral averaging interacts with the logarithm in large-dimensional rank-one Harish-Chandra--Itzykson--Zuber (HCIZ) integrals.
For independent bi-unitarily invariant matrices $A,B$ with uniformly bounded deterministic singular values and weakly convergent empirical spectral measures of their Gram matrices, the integral for $(BA)^*(BA)\otimes\Id_N$ has two limiting descriptions at speed $N^2$ for small real parameters: averaging its logarithm gives free multiplicative convolution, whereas taking the logarithm after averaging gives classical multiplicative convolution.
We also prove an exact finite-dimensional equality between this annealed integral and the one for $(A^*A)\otimes(B^*B)$.
For any fixed $k\ge3$ factors under the hypotheses stated for $A,B$, the annealed limits at speed $N^k$ generally differ for small positive parameters, and we characterize the equality cases.

For tensor products with any fixed number $k$ of factors, global Haar vectors at speed $N^k$ and product-state vectors at speed $N$ yield classical and free multiplicative convolution, respectively, for small real parameters under the corresponding hypotheses.

We also derive the known multiplicativity of the $S$-transform.
\end{abstract}

\section{Introduction}

The Harish-Chandra--Itzykson--Zuber (HCIZ) integral is a central tool in the theoretical-physics approach to random matrices.
Its exact evaluation originates in Harish-Chandra's work \cite{HarishChandra1957}, and its role in matrix models was developed by Itzykson and Zuber \cite{ItzyksonZuber1980}.
For Hermitian matrices $H,K\in M_N(\C)$, the integral
\begin{align*}
  \int_{\U(N)}\exp\{Nt\,\Tr(HUKU^*)\}\dd U
\end{align*}
integrates the relative angular degrees of freedom and leaves a function of the two spectra.
This separation makes it useful for studying eigenvalue distributions, external-source models \cite{ZinnJustin1997}, and operations on random matrices.
In particular, matrix-integral and saddle-point methods were used in \cite{ZinnJustin1999AddMult} to derive the free-probability formulas for addition and multiplication.
Tensor HCIZ integrals extend this program to integration over tensor products of unitary groups and to questions related to entanglement \cite{CollinsGurauLionniI,CollinsGurauLionniII}; cumulants and freeness for locally unitarily invariant random tensors are studied in \cite{CollinsGurauLionniIII,NechitaPark2025}.

In large dimension, the natural object is the logarithm of the integral, normalized by its exponential scale. This is the free energy in the statistical-mechanical interpretation, and its expansion organizes connected contributions.
The scale is typically $N^2$ for two matrices with macroscopic spectra \cite{Matytsin1994}, and $N$ when one matrix is a rank-one projection.
Rigorous links between spherical-integral free energies, spectral large deviations, and free probability were established in \cite{GuionnetZeitouni2002} and, in the rank-one regime, in \cite{GuionnetMaida2005}.
The free cumulants were identified coefficientwise in the rank-one expansion in \cite{Collins2003}.
Finite-rank asymptotics with outliers are treated in \cite{GuionnetHusson2022}; see \cite{HussonKo2025} for sublinear rank.

When the spectra are random, however, the angular integral is only one stage of the averaging. Even when the input eigenvalues in an additive problem, or the input singular values in a multiplicative problem, are deterministic, the spectrum relevant to the HCIZ integral of the sum or product can be random. It depends on the relative eigenvectors or singular vectors.
One must therefore specify whether spectral averaging is performed before or after taking the logarithm. To state the question in the rank-one setting, write
\begin{align*}
  I_N(t;H_N)=\E_x\exp\{Nt\,x^*H_Nx\},
\end{align*}
where $x$ is an independent Haar unit vector and the expectation is conditional on the Hermitian matrix $H_N$. The two deterministic quantities to compare are
\begin{align*}
  \frac1N\E_{H_N}\log I_N(t;H_N)
  \qquad\text{and}\qquad
  \frac1N\log\E_{H_N}I_N(t;H_N).
\end{align*}
The first is the disorder average of the quenched free energy, and the second is the annealed free energy; see \cite{Talagrand2003} and the precise definitions in \eqref{eq:quenched-spherical-integral}--\eqref{eq:annealed-quenched-jensen}.
Because $I_N$ depends only on the spectrum, the matrix expectation here is equivalently an average over the random empirical spectral measure.
Jensen's inequality orders these quantities, but does not identify their limits.
An atypical spectrum with exponentially small probability can still contribute to the annealed integral if the exponential weight compensates for its cost.
This competition is the content of the large-deviation and Laplace-principle approach to exponential averages; see \cite[Section~4.3]{DemboZeitouni2010} and \cite{GuionnetZeitouni2002}.
Quenched and annealed free energies of spherical integrals are compared in the study of large deviations for the largest eigenvalue of sums and products of invariant random matrices \cite{MergnyPotters2022LargeDeviations}.

\subsection*{Main results}

Let $A_N,B_N$ be independent bi-unitarily invariant matrices with deterministic, uniformly bounded singular values.
Suppose that the empirical spectral measures of $A_N^*A_N$ and $B_N^*B_N$ converge weakly to compactly supported probability measures $\mu_A,\mu_B$
on $[0,\infty)$.
Set $H_N=(B_NA_N)^*(B_NA_N)$.
By Corollary~\ref{cor:two-orders-of-averaging}, for all real $t$ sufficiently close to $0$,
\begin{align*}
  \lim_{N\to\infty}\frac1{N^2}\E_{A_N,B_N}\log I_{N^2}(t;H_N\otimes\Id_N)&=\int_0^tR_{\mu_A\boxtimes\mu_B}(s)\dd s,\\
  \lim_{N\to\infty}\frac1{N^2}\log\E_{A_N,B_N}I_{N^2}(t;H_N\otimes\Id_N)&=\int_0^tR_{\mu_A\circledast\mu_B}(s)\dd s.
\end{align*}
Here $R_\mu$ denotes the $R$-transform, and $\boxtimes$ and $\circledast$ denote free and classical multiplicative convolution, respectively. If neither $\mu_A$ nor $\mu_B$ is a point mass, the second limit is strictly larger than the first for all sufficiently small nonzero real $t$, by~\eqref{eq:conclusion-averaging-gap}.

The annealed limit follows from Theorems~\ref{thm:exact-two-leg-hciz} and~\ref{thm:global-classical}.
The former gives an exact finite-dimensional annealed identity between the global rank-one spherical integrals associated with the Gram matrices of the two operators
\begin{align*}
  X_{1,N}\otimes X_{2,N} \quad \text{and} \quad (X_{1,N}X_{2,N})\otimes\Id_N,
\end{align*}
where $X_{1,N}$ is arbitrary and $X_{2,N}$ is independent of $X_{1,N}$ and bi-unitarily invariant.
The identity holds for every $N$ and every real $t$, but need not hold conditional on the matrix realizations.

For every fixed $k\ge3$, Proposition~\ref{prop:higher-leg-aligned-spectrum} identifies the annealed limit for $(X_{1,N}\cdots X_{k,N})\otimes\Id_N^{\otimes(k-1)}$ at speed $N^k$ and small positive parameters, for independent bi-unitarily invariant matrices with deterministic, uniformly bounded singular values and weakly convergent empirical spectral measures of their Gram matrices.
Corollary~\ref{prop:higher-leg-counterexample} shows that, for sufficiently small positive parameters, this limit is strictly larger than that for $X_{1,N}\otimes\cdots\otimes X_{k,N}$ whenever none of the limiting measures is $\delta_0$ and at least two are not point masses.
In all other cases, the two limits agree.

We also compare the global Haar and product-state integrals for the same tensor product. For any fixed $k\ge2$ and matrices $X_{1,N},\ldots,X_{k,N}$, set $\mathcal X_N=X_{1,N}\otimes\cdots\otimes X_{k,N}$.
The empirical spectral measure of $\mathcal X_N^*\mathcal X_N$ is exactly the classical multiplicative convolution of the empirical spectral measures of $X_{i,N}^*X_{i,N}$.
Under the weak convergence and norm assumptions of Theorem~\ref{thm:global-classical}, the conditional normalized log-spherical integral for a global Haar vector at speed $N^k$ converges in probability to $\int_0^tR_{\mu_1\circledast\cdots\circledast\mu_k}(s)\dd s$ for real $t$ near $0$. With the additional edge convergence assumptions, the same limit holds on the larger domain stated there.
This statement requires neither independence nor invariance of the factors.
Under the hypotheses of Theorem~\ref{thm:local-free}, the annealed normalized log-spherical integral for a product-state vector $x_{1,N}\otimes\cdots\otimes x_{k,N}$ is governed at speed $N$ by $\mu_1\boxtimes\cdots\boxtimes\mu_k$.
For fixed matrices, the product-state integral is the tensor HCIZ integral over $\U(N)^k$ with external matrices $\mathcal X_N^*\mathcal X_N$ and a rank-one projection onto a fixed unit product vector, and exponential parameter $Nt$ \cite[Eqs.~(1.1)--(1.2)]{CollinsGurauLionniI}. The global Haar integral uses $\U(N^k)$; see also \cite{AubrunSzarek2017}.
For matrices with deterministic singular values as in Corollary~\ref{cor:local-free-deterministic}, the corresponding fixed-order cumulant asymptotics follow from \cite[Lemma~5.3]{CollinsGurauLionniII}. Their relation to tensorial free cumulants is discussed in \cite[Section~4.3]{BucDAlcheLionni2026}.

For multiplication, we use the following distributional identity.
If $A,B$ are independent bi-unitarily invariant matrices and $x,y$ are independent Haar unit vectors, independent of the matrices, then
\begin{align*}
  \norm{BAx}^2\stackrel{d}{=}\norm{Ax}^2\norm{By}^2,
\end{align*}
with independent factors on the right.
Related radial decompositions occur in the study of isotropic matrix products and Lyapunov exponents \cite{CohenNewman1984,Newman1986,KieburgKosters2019}.
The Laplace transform of this scalar product does not factor.
Moreover, its ordinary weak limit generally retains only the first moments of the limits of the empirical spectral measures of $A_N^*A_N$ and $B_N^*B_N$. Recovering the full spectral information requires the speed-$N$ logarithmic Laplace transform and a scalar saddle-point calculation.

At speed $N$, the spectral concentration and annealed HCIZ estimates of Propositions~\ref{prop:product-levy-concentration} and~\ref{prop:loewner-squeeze} justify averaging over the product spectrum under their stated hypotheses. Combining these estimates with~\eqref{eq:finite-radial-product} and a local form of the Laplace principle, Theorem~\ref{thm:radial-s-transform} gives a radial derivation of the classical identity
\begin{align*}
  S_{\mu_A\boxtimes\mu_B}(z)=S_{\mu_A}(z)S_{\mu_B}(z),
\end{align*}
proved in \cite[Theorem~2.6]{Voiculescu1987}. The argument uses neither a theorem on asymptotic freeness nor $S$-multiplicativity as an input. For the connection between random-matrix limits and free probability, see \cite{Voiculescu1991}.

Theorem~\ref{thm:quantitative-local-multiplicativity} gives a uniform finite-dimensional estimate for independent bi-unitarily invariant matrices $A_N,B_N$ with singular values in $[c,C]$, where $0<c<C<\infty$.
On a common complex disk about $0$, the multiplicative rule for the empirical $S$-transforms holds with uniform error at most any prescribed $\delta>0$, with probability at least $1-e^{-\gamma N^2}$ for all sufficiently large $N$, uniformly over the singular-value lists.

This route complements earlier matrix-integral approaches to multiplication.
The derivation in \cite[Section~3]{ZinnJustin1999AddMult} uses a character integral and a saddle-point calculation.
Multiplicative spherical integrals average powers of positive quadratic forms in the rank-one case; their asymptotics are studied by saddle-point analysis in \cite{MergnyPotters2020} and proved rigorously for orthogonal and unitary integrals of fixed rank in \cite{Husson2025}.
For deterministic positive singular values, taking power moments in~\eqref{eq:finite-radial-product} gives the rank-one case of the finite-dimensional product formula in \cite[Eq.~(12)]{MergnyPotters2020}.
Spherical transforms of polynomial ensembles linearize matrix multiplication at finite dimension \cite[Eq.~(2.20)]{KieburgKosters2019}.
Here we retain the ordinary additive rank-one HCIZ integral applied to $A_N^*A_N$ and justify its spectral averaging before applying~\eqref{eq:finite-radial-product}.
The rectangular spherical-integral interpretation of the rectangular $R$-transform offers another related perspective \cite{BenaychGeorges2011}.

A related motivation comes from Gurau's resolvent for real symmetric Gaussian tensors \cite{Gurau2020}, whose large-dimensional expansion involves Fuss--Catalan numbers. These also occur in the positive matrix-product measures obtained in our Marchenko--Pastur example \cite{BanicaBelinschiCapitaineCollins2011,PensonZyczkowski2011}.
Universality of the limiting moments for the symmetric tensor model is established in \cite{Bonnin2026}. The symmetric Wigner--Gurau limit differs from the Fuss--Catalan laws in our Marchenko--Pastur example.

\subsection*{Organization of the paper}

The paper is organized as follows.
Section~\ref{sec:notation} introduces the rank-one spherical integral, the two orders of averaging, and the distributional identity~\eqref{eq:finite-radial-product}.
Section~\ref{sec:radial-s-transform} proves the spectral comparison and concentration estimates, gives another proof of the multiplicativity of the $S$-transform, and obtains a uniform finite-dimensional estimate for that relation.
Section~\ref{sec:two-leg} proves the exact equality of annealed integrals for products of two matrices stated in Theorem~\ref{thm:exact-two-leg-hciz}.
Section~\ref{sec:tensor-regimes} compares global Haar and product-state vectors and compares averaging before and after the logarithm.
Section~\ref{sec:higher-leg-obstruction} determines the annealed limit for products of three or more matrices and shows that equality with the tensor-product limit fails in general for small positive parameters.

\section{Preliminaries}
\label{sec:notation}

We write $M_N(\C)$ for the algebra of $N\times N$ complex matrices and define the normalized trace by
$\tr_N(A)=\frac1N\Tr(A)$.
For matrices, $\norm{\cdot}$ denotes the operator norm and $\norm{A}_{\mathrm{HS}}=(\Tr(A^*A))^{1/2}$ the unnormalized Hilbert--Schmidt norm. For vectors, $\norm{\cdot}$ is the Euclidean
norm.
We write $\Id_N$ for the $N\times N$ identity matrix.
The unit sphere of $\C^N$ is denoted by $\Sph^{2N-1}$ and is always equipped with its Haar probability measure.
A random matrix $A\in M_N(\C)$ is called \emph{bi-unitarily invariant} if, for all deterministic $U,V\in\U(N)$,
\begin{align*}
  A\stackrel{d}{=}UAV.
\end{align*}
Equivalently, $A$ admits the distributional representation
$A\stackrel{d}{=}UDV^*$,
where $D=\operatorname{diag}(s_1,\ldots,s_N)$, with $s_1\geq\cdots\geq s_N\geq0$, is the random diagonal matrix of singular values, and $U,V$ are independent Haar-distributed unitaries, independent of $D$.

For a Hermitian matrix $H\in M_N(\C)$ we denote by
\begin{align*}
  \mu_H=\frac1N\sum_{i=1}^N\delta_{\lambda_i(H)}
\end{align*}
its empirical spectral measure.

For probability measures $\mu,\nu$ on $\R$, with distribution functions $F_\mu,F_\nu$, their L\'evy distance is
\begin{align}\label{eq:levy-distance-definition}
  d_{\mathrm L}(\mu,\nu)
  =\inf\left\{\varepsilon>0:
  F_\mu(x-\varepsilon)-\varepsilon\leq F_\nu(x)
  \leq F_\mu(x+\varepsilon)+\varepsilon
  \ \text{for every }x\in\R\right\}.
\end{align}
It metrizes weak convergence.
We call a probability measure \emph{non-degenerate} if it is not a point mass.

For Hermitian matrices $H,K\in M_N(\C)$, $H\leq K$ means that $K-H$ is positive semidefinite.

\subsection{Transforms and convolutions}

We also recall the free-probability transforms used below. Let $\mu$ be a compactly supported probability measure. Its Cauchy transform
\begin{align}\label{eq:cauchy-transform}
  G_\mu(z)=\int\frac1{z-x}\,\dd\mu(x)
\end{align}
is univalent on a neighborhood of infinity and satisfies $G_\mu(z)\sim z^{-1}$. We write $K_\mu$ for the inverse branch defined for sufficiently small nonzero complex $w$, characterized by $G_\mu(K_\mu(w))=w$ and $K_\mu(w)\sim w^{-1}$ as $w\to0$.
The $R$-transform is the function analytic in a neighborhood of $0$ defined by
\begin{align}\label{eq:r-transform-definition}
  R_\mu(w)=K_\mu(w)-\frac1w.
\end{align}
Since $G_\mu(z)=z^{-1}+m_1(\mu)z^{-2}+O(z^{-3})$ as $z\to\infty$, inversion gives $K_\mu(w)=w^{-1}+m_1(\mu)+O(w)$ as $w\to0$. Hence its value at $0$ is defined by continuity as $R_\mu(0)=m_1(\mu)$.
Equivalently,
\begin{align}\label{eq:r-cumulant-expansion}
  R_\mu(w)=\sum_{n\geq1}\kappa_n(\mu)w^{n-1},
\end{align}
where $\kappa_n(\mu)$ are the free cumulants.

We also use the real continuation appearing in rank-one HCIZ asymptotics. Let $\ell_\mu=\min\supp\mu$ and $r_\mu=\max\supp\mu$. For any $\ell\leq\ell_\mu$ and $r\geq r_\mu$, set
\begin{align}\label{eq:real-r-transform-domain}
  \mathcal D_\mu^{\ell,r}=G_\mu((-\infty,\ell)\cup(r,\infty))\cup\{0\},
  \qquad
  \mathcal D_\mu=\mathcal D_\mu^{\ell_\mu,r_\mu}.
\end{align}
The restrictions of $G_\mu$ to $(-\infty,\ell_\mu)$ and $(r_\mu,\infty)$ are invertible, and $\mathcal D_\mu^{\ell,r}\subset\mathcal D_\mu$. On $\mathcal D_\mu\setminus\{0\}$, let $K_\mu$ denote the corresponding inverse branches and set $R_\mu(w)=K_\mu(w)-w^{-1}$, with $R_\mu(0)=m_1(\mu)$. This real continuation agrees with the analytic function in~\eqref{eq:r-transform-definition} for real arguments sufficiently close to $0$.

If $\mu$ is supported on $[0,\infty)$ and has non-zero first moment, set
\begin{align*}
  \psi_\mu(z)=\int\frac{zx}{1-zx}\,\dd\mu(x)
  =\sum_{n\geq1}m_n(\mu)z^n,
  \qquad
  \chi_\mu=\psi_\mu^{-1}.
\end{align*}
Here $\chi_\mu$ is the local analytic inverse for composition, with $\chi_\mu(0)=0$.
The $S$-transform is the function analytic in a neighborhood of $0$ defined by
\begin{align*}
  S_\mu(z)=\frac{1+z}{z}\,\chi_\mu(z).
\end{align*}
Since $\psi_\mu'(0)=m_1(\mu)\neq0$, differentiation of $\psi_\mu(\chi_\mu(z))=z$ at $0$ gives $\chi_\mu'(0)=1/m_1(\mu)$. Using $\chi_\mu(0)=0$, the continuous extension therefore satisfies $S_\mu(0)=\chi_\mu'(0)=1/m_1(\mu)$.
These definitions follow Voiculescu's convention
\cite{Voiculescu1987,VoiculescuDykemaNica1992,NicaSpeicher2006,MingoSpeicher2017}.

We shall use the equivalent relation
\begin{align}\label{eq:s-r-coordinate-relation}
  S_\mu\bigl(uR_\mu(u)\bigr)=\frac1{R_\mu(u)},
\end{align}
valid for $u$ near $0$ whenever the first moment of $\mu$ is non-zero.
Indeed, if $z=\chi_\mu(w)$, then
$G_\mu(1/z)=z(1+w)=wS_\mu(w)$,
where the last equality follows from
$S_\mu(w)=(1+w)z/w$. Hence
$K_\mu\bigl(wS_\mu(w)\bigr)=\frac1z=\frac{1+w}{wS_\mu(w)}$,
and therefore
\begin{align*}
  R_\mu\bigl(wS_\mu(w)\bigr)=\frac1{S_\mu(w)}.
\end{align*}
Since $S_\mu(0)=1/m_1(\mu)\neq0$, the map $w\mapsto wS_\mu(w)$ is
locally invertible at $0$.  The preceding identity shows that its local
inverse is $u\mapsto uR_\mu(u)$, which gives
\eqref{eq:s-r-coordinate-relation}.

\begin{definition}[multiplicative free convolution]\label{def:boxtimes}
Let $\mu_A$ and $\mu_B$ be compactly supported probability measures on $[0,\infty)$.
Choose positive $*$-free noncommutative random variables $x$ and $y$ with laws $\mu_A$ and $\mu_B$ in a noncommutative probability space with tracial state $\tau$.
The multiplicative free convolution $\mu_A\boxtimes\mu_B$ is the law of the positive variable
\begin{align}\label{eq:boxtimes-operator-definition}
  \mu_A\boxtimes\mu_B
  := \Law\!\left(x^{1/2}yx^{1/2}\right).
\end{align}
Equivalently, its moments are $\tau((xy)^n)$, since these agree by traciality with the moments of the variable defining the law in
\eqref{eq:boxtimes-operator-definition}.
By associativity of multiplicative free convolution, the notation $\mu_1\boxtimes\cdots\boxtimes\mu_k$ is unambiguous.
\end{definition}

\begin{definition}[classical multiplicative convolution]
If $\mu$ and $\nu$ are probability measures on $[0,\infty)$, their classical multiplicative convolution $\mu\circledast\nu$ is the law of $XY$, where $X$ and $Y$ are independent scalar random variables with laws $\mu$ and $\nu$.
We write $\mu_1\circledast\cdots\circledast\mu_k$ for the iterated classical multiplicative convolution.
\end{definition}

\subsection{Rank-one spherical integrals}

\begin{definition}[rank-one spherical integral]
Let $A\in M_N(\C)$, let $H=H^*\in M_N(\C)$, and let $x$ be Haar distributed on $\Sph^{2N-1}$.
The rank-one spherical integral is
\begin{align}\label{eq:rank-one-spherical-integral}
  I_N(t;H)=\E_x\exp\{Nt\,x^*Hx\},
\end{align}
which is finite for every real $t$. If $A$ is random and independent of $x$, the \emph{quenched} normalized log-spherical integral is the conditional quantity
\begin{align}\label{eq:quenched-spherical-integral}
  \Lambda_{A,N}^{\mathrm q}(t)
  =\frac1N\log I_N(t;A^*A),
\end{align}
which remains random through the matrix realization. The \emph{annealed} rank-one spherical integral and its normalized logarithm are
\begin{align}\label{eq:annealed-spherical-integral}
  J_N(t;A)=\E_{A,x}\exp\{Nt\,x^*A^*Ax\},
  \qquad
  \Lambda_{A,N}^{\mathrm a}(t)=\frac1N\log J_N(t;A).
\end{align}
This annealed expression is used only for parameters for which it is finite.
\end{definition}

\begin{remark}\label{rem:rank-one-hciz-representation}
For fixed $A$, $I_N(\cdot;A^*A)$ determines the spectrum of $A^*A$ by Proposition~\ref{prop:finite-reconstruction}.
The terminology follows the standard convention for random systems with disorder; see, for example,~\cite{Talagrand2003}. The matrix is held fixed in the quenched quantity \eqref{eq:quenched-spherical-integral}, whereas it is averaged inside the logarithm in the annealed quantity \eqref{eq:annealed-spherical-integral}.
One also considers the disorder-averaged quenched free energy $\E_A\Lambda_{A,N}^{\mathrm q}(t)$. Jensen's inequality gives
\begin{align}\label{eq:annealed-quenched-jensen}
  \E_A\Lambda_{A,N}^{\mathrm q}(t)
  \leq \Lambda_{A,N}^{\mathrm a}(t).
\end{align}
The two sides can differ because rare matrix realizations may be visible to the annealed exponential moment. In particular, $J_N(\cdot;A)=\E_A I_N(\cdot;A^*A)$ need not determine the law of the random empirical spectral measure of $A^*A$.

Also, the normalization $Nt$ in~\eqref{eq:rank-one-spherical-integral} is the natural rank-one HCIZ normalization.
If $X_N=ee^*$ is a rank-one orthogonal projector with $\norm{e}=1$ and $Y_N$ is Hermitian, the corresponding HCIZ integral is
\begin{align*}
  \int_{\U(N)}\exp\{Nt\Tr(X_NUY_NU^*)\}\dd U.
\end{align*}
Since $\Tr(X_NUY_NU^*)=(U^*e)^*Y_N(U^*e)$ and $U^*e$ is a Haar unit vector, this integral equals $I_N(t;Y_N)$.
\end{remark}

The moment formula used below is also given in \cite{DunklEtAl2011NumericalShadows}.

\begin{proposition}
\label{prop:finite-reconstruction}
Let $H\geq0$ have eigenvalues $\lambda_1,\ldots,\lambda_N$. The law of $x^*Hx$, with $x$ Haar on $\Sph^{2N-1}$, determines the unordered $N$-tuple $(\lambda_1,\ldots,\lambda_N)$.
Equivalently, the function $t\mapsto I_N(t;H)$ determines that unordered tuple.
\end{proposition}

\begin{proof}
After diagonalizing $H$ and using Haar invariance,
$x^*Hx\stackrel{d}{=}\sum_{i=1}^N\lambda_i|x_i|^2$, where
$(|x_1|^2,\ldots,|x_N|^2)$ has the Dirichlet distribution with parameters
$(1,\ldots,1)$. Its moment formula gives, for every $k\geq0$,
\begin{align*}
  \E_x(x^*Hx)^k=\frac{k!(N-1)!}{(N+k-1)!}h_k(\lambda_1,\ldots,\lambda_N),
\end{align*}
where $h_k$ is the complete homogeneous symmetric polynomial of degree
$k$. Hence the law of $x^*Hx$ determines $h_k(\lambda_1,\ldots,\lambda_N)$
for every $k$. Since
\begin{align*}
  \sum_{k\geq0}h_k(\lambda_1,\ldots,\lambda_N)z^k
  =\prod_{i=1}^N\frac1{1-z\lambda_i},
\end{align*}
these coefficients determine the polynomial
$\prod_{i=1}^N(1-z\lambda_i)$ and therefore the eigenvalue multiset,
including the multiplicity of $0$ since $N$ is fixed. Finally,
$x^*Hx$ is bounded and
$\left.\frac{d^k}{dt^k}I_N(t;H)\right|_{t=0}
=N^k\E_x(x^*Hx)^k$, so $I_N(\cdot;H)$ determines the same moments.
\end{proof}

The next statement gives the large-$N$ connection with free probability.

\begin{theorem}[rank-one HCIZ and the \texorpdfstring{$R$}{R}-transform]
\label{thm:rank-one-hciz}
This is the unitary rank-one case of \cite[Theorem~2]{GuionnetMaida2005}, recorded here with our normalization.
Let $H_N=H_N^*$ be deterministic matrices whose empirical spectral measures converge weakly to a compactly supported probability measure $\mu$.
Assume that
\begin{align*}
  \lambda_{\min}(H_N)\to \ell,
  \qquad
  \lambda_{\max}(H_N)\to r,
\end{align*}
where $\supp\mu\subset[\ell,r]$.
Then, for every real $t\in\mathcal D_\mu^{\ell,r}$ as defined in~\eqref{eq:real-r-transform-domain}, we have
\begin{align}\label{eq:rank-one-hciz-r-primitive}
  \lim_{N\to\infty}\frac1N\log I_N(t;H_N)
  =\int_0^tR_\mu(s)\dd s.
\end{align}
In particular, this holds for all real $t$ in a neighborhood of $0$.
\end{theorem}

\begin{remark}
In the notation of \cite{GuionnetMaida2005},
$\beta=2$, so the factor $\frac\beta2\int_0^{2t/\beta}R_\mu(s)\dd s$ is the right-hand side of \eqref{eq:rank-one-hciz-r-primitive}.
The cumulant form is closely related to rank-one Itzykson--Zuber cumulants in \cite{Collins2003}.
\end{remark}

\subsection{Matrix products}
\label{sec:rank-one-products}

\begin{proposition}
\label{prop:finite-radial-identity}
Let $A,B\in M_N(\C)$ be independent bi-unitarily invariant random matrices, and let $x,y$ be independent Haar unit vectors, independent of $A,B$. Then
\begin{align}
  \label{eq:finite-radial-product}
  \norm{BAx}^2
  \stackrel{d}{=}
  \norm{Ax}^2\,\norm{By}^2,
\end{align}
with the two factors on the right-hand side independent.
\end{proposition}

\begin{proof}
Since $A$ is left-unitarily invariant and independent of $x$, the vector $Ax$ is unitarily invariant. Conditionally on its norm, it is uniformly distributed on the sphere of that radius. Since $B$ is independent of $(A,x)$, we may therefore write
$(B,Ax)\stackrel{d}{=}(B,\norm{Ax}y)$, where $y$ is Haar distributed and independent of $(B,\norm{Ax})$. Applying $B$ and taking squared norms proves the stated identity, and the two factors on the right are independent.
\end{proof}

\begin{remark}[one-sided invariance suffices]
The proof only uses left-unitary invariance of $A$. In fact, $x$ need not be Haar distributed, and $B$ may be any random matrix independent of $(A,x)$.
We nevertheless retain bi-unitary invariance in the statement, in accordance with the assumptions used in the subsequent results.
\end{remark}

\section{Spherical integrals and the \texorpdfstring{$S$}{S}-transform}
\label{sec:radial-s-transform}

\subsection{Laplace asymptotics}

We now explain how the preceding identity in Proposition~\ref{prop:finite-radial-identity} gives the multiplicativity of the $S$-transform.
Whenever the following limit exists on an interval about $0$, we denote it by
\begin{align}\label{eq:normalized-radial-log-laplace}
  \Phi_A(t)=\lim_{N\to\infty}\frac1N\log J_N(t;A)=\lim_{N\to\infty}\frac1N\log\E_{A_N}I_N(t;A_N^*A_N).
\end{align}
If $A_N$ is deterministic, then $J_N(t;A)=I_N(t;A_N^*A_N)$.
Proposition~\ref{prop:loewner-squeeze} below supplies sufficient hypotheses for this annealed limit to exist and then gives
\begin{align*}
  \Phi_A(t)=\int_0^t R_A(s)\,\dd s,
\end{align*}
where $R_A$ is the $R$-transform of the limit of the empirical spectral measure of $A_N^*A_N$.
For deterministic spectra this is the local form of Theorem~\ref{thm:rank-one-hciz}.

We shall use the following local form of the Laplace principle. Here ``local'' refers to a neighborhood of $0$ for the log-Laplace parameter and the corresponding range of means; see, for example, \cite{DemboZeitouni2010}. We include the short proof for completeness.

\begin{proposition}
\label{prop:local-laplace}
Let $Q_N$ be random variables with values in $[0,C]$.
Assume that there is an open interval $U\subset\R$ containing $0$ such that
\begin{align}\label{eq:local-laplace-scaled-cgf}
  L_N(s)=\frac1N\log\E e^{NsQ_N},
\end{align}
converges locally uniformly on $U$ to a $C^2$ strictly convex function $L$. Let
\begin{align*}
  \mathcal I(x)=\sup_{s\in U}\{sx-L(s)\},
\end{align*}
be the local Legendre transform.
Let $J$ be a compact subinterval of $L'(U)$. If $f\colon[0,C]\to\R$ is continuous and the function $x\mapsto f(x)-\mathcal I(x)$ attains its maximum at a point of $J$, then
\begin{align}\label{eq:local-laplace-variational-limit}
  \lim_{N\to\infty}\frac1N\log\E e^{Nf(Q_N)}
  =\sup_{x\in[0,C]}\{f(x)-\mathcal I(x)\}.
\end{align}

The same conclusion holds if $Nf(Q_N)$ is replaced by $Nf(Q_N)+N\varepsilon_N(Q_N)$, where $\sup_{x\in[0,C]}|\varepsilon_N(x)|\to0$.
In particular, let $f_t\colon[0,C]\to\R$ be a family of continuous functions indexed by a real parameter $t$ near $0$. Suppose that $L''(0)>0$, that $J$ contains $L'(0)$ in its interior, and that $f_t\to f_0$ uniformly on $[0,C]$ as $t\to0$, where $f_0$ is constant.
Then every maximizer of $f_t-\mathcal I$ lies in $J$ for all sufficiently small $|t|$.
If, in addition, $L$ and $(t,x)\mapsto f_t(x)$ are analytic near $(0,L'(0))$, the maximizer in $J$ is unique and depends analytically on $t$ after a further restriction of the parameter interval.
\end{proposition}

\begin{proof}

Let $x=L'(s)$ with $s\in U$. For each $N$, define a probability measure $\widetilde{\mathbb{P}}_s$ by
\begin{align*}
\frac{\mathrm d\widetilde{\mathbb{P}}_s}{\mathrm d\mathbb{P}}=\frac{e^{NsQ_N}}{\mathbb{E}e^{NsQ_N}}.
\end{align*}
Write $\widetilde{\mathbb{E}}_s$ for expectation under $\widetilde{\mathbb{P}}_s$, and then under this measure, the log-Laplace transform of $Q_N$ is
\begin{align*}
\frac1N\log\widetilde{\mathbb{E}}_s e^{NrQ_N}=L_N(s+r)-L_N(s).
\end{align*}
By local uniform convergence, this converges near $0$ to $L(s+r)-L(s)$, whose derivative at $0$ is $x$.

Fix $\varepsilon>0$ and $r>0$ small enough that $s-r,s+r\in U$. Markov's inequality applied to $e^{NrQ_N}$ and $e^{-NrQ_N}$ gives
\begin{align*}
\widetilde{\mathbb{P}}_s(Q_N\ge x+\varepsilon)&\le e^{-Nr(x+\varepsilon)}\widetilde{\mathbb{E}}_s e^{NrQ_N}
=\exp\{N(L_N(s+r)-L_N(s)-r(x+\varepsilon))\},\\
\widetilde{\mathbb{P}}_s(Q_N\le x-\varepsilon)&\le e^{Nr(x-\varepsilon)}\widetilde{\mathbb{E}}_s e^{-NrQ_N}
=\exp\{N(L_N(s-r)-L_N(s)+r(x-\varepsilon))\}.
\end{align*}
Since $x=L'(s)$, both limiting exponents are strictly negative for $r>0$ sufficiently small. 
Local uniform convergence of $L_N$ therefore gives some $a>0$ such that $\widetilde{\mathbb{P}}_s(|Q_N-x|\ge\varepsilon)\le2e^{-aN}$ for all large $N$.
Returning to the original law,
\begin{align*}
  \mathbb P(|Q_N-x|<\varepsilon)
  \ge\exp\{NL_N(s)-Nsx-N|s|\varepsilon\}(1-2e^{-aN}),
\end{align*}
and hence
\begin{align*}
  \liminf_{N\to\infty}\frac1N\log\mathbb P(|Q_N-x|<\varepsilon)
  \ge-sx+L(s)-|s|\varepsilon.
\end{align*}
Since $x=L'(s)$, convexity gives $\mathcal I(x)=sx-L(s)$.
Letting $\varepsilon\downarrow0$ proves the local lower bound $-\mathcal I(x)$ for $x\in L'(U)$.

For the upper bound, set $\mathcal M_f=\sup_{x\in[0,C]}\{f(x)-\mathcal I(x)\}$ and fix $\eta>0$.
By the definition of $\mathcal I$ and continuity of $f$, the sets
\begin{align*}
  \{x\in[0,C]:f(x)-sx+L(s)<\mathcal M_f+\eta\},\qquad s\in U,
\end{align*}
form an open cover of $[0,C]$. Choose $s_1,\ldots,s_m\in U$ giving a finite subcover. Then
\begin{align*}
  \E e^{Nf(Q_N)}\le\sum_{i=1}^m\exp\{N(\mathcal M_f+\eta+L_N(s_i)-L(s_i))\}.
\end{align*}
Since each $s_i$ is fixed, convergence of $L_N$ gives
\begin{align*}
  \limsup_{N\to\infty}\frac1N\log\E e^{Nf(Q_N)}\le\mathcal M_f+\eta.
\end{align*}
Letting $\eta\downarrow0$ gives the upper bound. 
At a maximizer $x_*\in J$, the local lower bound obtained above and continuity of $f$ give
\begin{align*}
  \liminf_{N\to\infty}\frac1N\log\E e^{Nf(Q_N)}\ge f(x_*)-\mathcal I(x_*)=\mathcal M_f.
\end{align*}
This proves~\eqref{eq:local-laplace-variational-limit}. The perturbation clause follows by squeezing between $f-\|\varepsilon_N\|_\infty$ and $f+\|\varepsilon_N\|_\infty$.

For the family assertion, $x_0=L'(0)$ is the unique zero of $\mathcal I$. Lower semicontinuity and compactness give $\inf_{[0,C]\setminus J}\mathcal I>0$.
Uniform convergence of $f_t$ to the constant $f_0$ therefore forces every maximizer into $J$ for small $|t|$.
In the analytic case, local Legendre duality gives $\mathcal I''(x_0)=1/L''(0)>0$. Since $f_0$ is constant, $\partial_x^2(f_0-\mathcal I)(x_0)<0$. The analytic implicit-function theorem applied to $\partial_x(f_t-\mathcal I)=0$ gives a unique analytic critical point near $x_0$, which remains a local maximizer for small $|t|$. After shrinking $J$ if necessary, the family assertion makes it the unique global maximizer.
\end{proof}

\begin{remark}
Pointwise convergence of the convex functions $L_N$ on $U$ already implies local uniform convergence on compact subintervals of $U$.
\end{remark}

\subsection{Spectral concentration}

We first give a sufficient condition for averaging over the random spectrum to preserve the limiting normalized logarithm of a rank-one spherical integral. We then show that the required spectral concentration is preserved under products of independent bi-unitarily invariant matrices with uniformly bounded norms.

\begin{proposition}[annealed rank-one spherical integrals]
\label{prop:loewner-squeeze}
Let $H_N$ be random positive matrices with $\norm{H_N}\le C$ almost surely, and let $\nu$ be a compactly supported probability measure on $[0,C]$.
Let $x_N$ be a Haar unit vector independent of $H_N$.
Assume that, for every $\varepsilon>0$, the empirical spectral measure of $H_N$ satisfies
\begin{align}\label{eq:loewner-squeeze-spectral-concentration}
  \limsup_{N\to\infty} \frac1N\log \mathbb P\bigl(d_{\mathrm L}(\mu_{H_N},\nu)>\varepsilon\bigr)
  =-\infty,
\end{align}
where $d_{\mathrm L}$ is defined in \eqref{eq:levy-distance-definition}.
Then
\begin{align}\label{eq:loewner-squeeze-annealed-limit}
  \lim_{N\to\infty} \frac1N\log \E_{H_N,x_N}\exp\{Nt\,x_N^*H_Nx_N\}
  =\lim_{N\to\infty} \frac1N\log \E_{H_N} I_N (t; H_N)
  =\int_0^t R_\nu(s)\dd s,
\end{align}
for all real $|t|<1/(64C)$, with the trivial interpretation when $C=0$.
\end{proposition}

\begin{proof}
The proposition follows from the rank-one asymptotics and continuity estimates in \cite{GuionnetMaida2005}, with the contribution of exceptional spectra controlled by the assumed superexponential concentration.
The cases $C=0$ and $t=0$ are immediate. Fix $C>0$ and $0<|t|<1/(64C)$.
For every probability measure $\rho$ supported in $[0,C]$, the bounds $G_\rho(2C)\ge1/(2C)$ and $G_\rho(-C)\le-1/(2C)$ ensure that $K_\rho(t)$ is at distance greater than $C$ from $[0,C]$, so the parameter condition in \cite[Theorem~2 and Lemma~14(3)]{GuionnetMaida2005} holds uniformly with $\beta=2$ and $\eta=C$.
Since the Dudley and L\'evy distances are uniformly equivalent on probability measures supported in $[0,C]$, these results imply that, for each $a>0$, there exists $\varepsilon>0$ such that $\left|N^{-1}\log I_N(t;H)-\int_0^tR_\nu(s)\dd s\right|\le a$ for all sufficiently large $N$ and all $0\le H\le C\Id_N$ with $d_{\mathrm L}(\mu_H,\nu)\le\varepsilon$.

Put $p_N=\PP(d_{\mathrm L}(\mu_{H_N},\nu)>\varepsilon)$. Splitting the expectation and using $I_N(t;H_N)\le e^{N|t|C}$ gives
\begin{align*}
  (1-p_N)e^{N\left(\int_0^tR_\nu(s)\dd s-a\right)}\le\E I_N(t;H_N)\le e^{N\left(\int_0^tR_\nu(s)\dd s+a\right)}+p_Ne^{N|t|C}.
\end{align*}
By~\eqref{eq:loewner-squeeze-spectral-concentration}, the last term is superexponentially small at speed $N$. Taking normalized logarithmic limits and then letting $a\downarrow0$ proves the claim.
\end{proof}

The asymptotic exchange of spectral averaging and the logarithm also arises in the spherical-integral approach to additive free convolution. For deterministic Hermitian matrices $H_N,K_N$ and Haar $U_N\in\U(N)$, Haar invariance gives, for every real $t$,
\begin{align*}
  \E_{U_N}I_N(t;H_N+U_NK_NU_N^*)=I_N(t;H_N)I_N(t;K_N).
\end{align*}
This is the rank-one analogue of the factorization used in \cite[Section~2]{ZinnJustin1999AddMult}. Identifying the normalized logarithm of this annealed integral with the free energy of a typical sum requires an asymptotic exchange of integration and logarithm; see \cite[Section~6.2]{GuionnetMaida2005} for the small-parameter argument under the stated hypotheses and Section~1.1 of the same reference for the discussion of its unitary extension.

\begin{proposition}[spectral concentration for matrix products]\label{prop:product-levy-concentration}
Let $A_N$ and $B_N$ be independent bi-unitarily invariant matrices with $\norm{A_N},\norm{B_N}\le C$ almost surely, where $C<\infty$.
Let $\mu_A,\mu_B$ be compactly supported probability measures on $[0,C^2]$. Assume that, for every $\varepsilon>0$,
\begin{align}\label{eq:product-input-spectral-concentration}
  \limsup_{N\to\infty}\frac1N\log
  \mathbb P\bigl(d_{\mathrm L}(\mu_{A_N^*A_N},\mu_A)>\varepsilon\bigr)
  &=-\infty,\\
  \limsup_{N\to\infty}\frac1N\log
  \mathbb P\bigl(d_{\mathrm L}(\mu_{B_N^*B_N},\mu_B)>\varepsilon\bigr)
  &=-\infty.
\end{align}
Then the empirical spectral measure of $(B_NA_N)^*(B_NA_N)$ satisfies
\begin{align}\label{eq:product-output-spectral-concentration}
  \limsup_{N\to\infty}
  \frac1N\log
  \mathbb P\bigl(
  d_{\mathrm L}(\mu_{(B_NA_N)^*(B_NA_N)},\mu_A\boxtimes\mu_B)>\varepsilon
  \bigr)
  =-\infty,
\end{align}
for every $\varepsilon>0$. The same conclusion holds with $A_NB_N$ in place of $B_NA_N$.
\end{proposition}

\begin{proof}
The case $C=0$ is immediate, so assume $C>0$.
Condition on the singular values. By independence and bi-unitary invariance, we may realize the conditional law of $(A_N,B_N)$ so that
\begin{align*}
  A_N=U_{A,N}D_{A,N}V_{A,N}^*,
  \qquad
  B_N=U_{B,N}D_{B,N}V_{B,N}^*,
\end{align*}
where, conditionally on $D_{A,N}$ and $D_{B,N}$, the four unitary matrices are independent Haar unitaries.
Under this realization, the matrix
$W_N:=V_{B,N}^*U_{A,N}$
is Haar and
\begin{align*}
  (B_NA_N)^*(B_NA_N)
  = V_{A,N}H_NV_{A,N}^*,
  \qquad
  H_N:=D_{A,N}W_N^*D_{B,N}^2W_ND_{A,N}.
\end{align*}
Since
$\mu_{D_{A,N}^2}=\mu_{A_N^*A_N}$ and $\mu_{D_{B,N}^2}=\mu_{B_N^*B_N}$, the conditioned input measures are precisely the empirical spectral measures of $A_N^*A_N$ and $B_N^*B_N$ appearing in the statement.

Fix $\varepsilon>0$. The input measures are supported in $[0,C^2]$, and the empirical spectral measures of $(B_NA_N)^*(B_NA_N)$ are supported in $[0,C^4]$.
On a fixed compact interval, weak convergence is equivalent to convergence of all moments.
Moreover, each moment of $\sigma\boxtimes\tau$ is a universal polynomial in finitely many moments of $\sigma$ and $\tau$.
Compactness of the spaces of probability measures on these intervals therefore gives $\delta>0$, an integer $K\ge1$, and $\eta>0$ such that for $\sigma,\tau$ supported in $[0,C^2]$,
\begin{align}\label{eq:free-product-continuity-choice}
  d_{\mathrm L}(\sigma,\mu_A)\le\delta,\quad
  d_{\mathrm L}(\tau,\mu_B)\le\delta
  \quad\Longrightarrow\quad
  d_{\mathrm L}(\sigma\boxtimes\tau,\mu_A\boxtimes\mu_B)
  <\frac{\varepsilon}{2},
\end{align}
and, for any $\rho,\rho'$ supported in $[0,C^4]$,
\begin{align}\label{eq:finitely-many-moments-control-levy}
  \max_{1\le j\le K}|m_j(\rho)-m_j(\rho')|<\eta
  \quad\Longrightarrow\quad
  d_{\mathrm L}(\rho,\rho')<\frac{\varepsilon}{2}.
\end{align}
If~\eqref{eq:finitely-many-moments-control-levy} failed for every $K,\eta$, compactness would give two limiting measures with identical moments but L\'evy distance at least $\varepsilon/2$, contradicting moment determinacy on $[0,C^4]$.

For every fixed $j$, the finite-$N$ unitary Weingarten formula
\cite[Corollaries~2.4 and~2.7]{CollinsSniady2006Haar} gives the uniform estimate
\begin{align}\label{eq:uniform-haar-moment-centering}
  \sup_{\substack{D,E\geq0\ \mathrm{diagonal}\\ \norm D,\norm E\leq C}}
  \left|
  \E_W\tr_N\bigl((DW^*E^2WD)^j\bigr) -m_j(\mu_{D^2}\boxtimes\mu_{E^2})
  \right|
  \leq \frac{C_{j,C}}{N^2},
\end{align}
for all sufficiently large $N$.
To check the normalization and uniformity, put $\gamma=(1,2,\ldots,j)$ and write $\tr_\pi(M)=\prod_{c\text{ cycle of }\pi}\tr_N(M^{|c|})$. The unitary Weingarten formula gives
\begin{align*}
  \E_W\tr_N\left((DW^*E^2WD)^j\right)
  =\sum_{\sigma,\tau\in\mathfrak S_j}N^{\#\sigma+\#(\tau^{-1}\gamma)-1}\mathrm{Wg}_N(\sigma^{-1}\tau)\tr_\sigma(D^2)\tr_{\tau^{-1}\gamma}(E^2).
\end{align*}
Here $\#\pi$ is the number of cycles and $|\pi|=j-\#\pi$.
By \cite{CollinsSniady2006Haar},
$\mathrm{Wg}_N(\pi)=N^{-j-|\pi|}\left(\operatorname{Moeb}(\pi)+O_j(N^{-2})\right).$
The total exponent of $N$ is $j-1-|\sigma|-|\sigma^{-1}\tau|-|\tau^{-1}\gamma|$, a nonpositive even integer. It vanishes precisely for the geodesic chains $\mathrm{id}\le\sigma\le\tau\le\gamma$, identified with nested non-crossing partitions. Summing first over $\sigma\le\tau$ is M\"obius inversion of moments into free cumulants; the remaining sum is the moment formula for the product of free variables. Thus the terms of order $N^0$ give $m_j(\mu_{D^2}\boxtimes\mu_{E^2})$. Every normalized trace product is bounded by $C^{4j}$, and all other contributions are $O_{j,C}(N^{-2})$. This proves the uniform estimate without using $S$-multiplicativity.

It remains to control fluctuations about the centers in \eqref{eq:uniform-haar-moment-centering}. The telescoping identity for two $j$-th powers and the inequalities $\norm{D W^*E^2WD}\le C^4$ and $\norm{W^*E^2W-W'^*E^2W'}_{\mathrm{HS}} \le2C^2\norm{W-W'}_{\mathrm{HS}}$ show that
\begin{align}\label{eq:haar-moment-lipschitz}
  W\longmapsto\tr_N\bigl((DW^*E^2WD)^j\bigr)
  \quad\text{is}\quad
  2jC^{4j}N^{-1/2}\text{-Lipschitz}
\end{align}
for the Hilbert--Schmidt distance.
Haar concentration on $\U(N)$ for Lipschitz functions  \cite[Theorem~5.17]{Meckes2019} therefore yields
\begin{align}
  \mathbb P_W\left(\left|\tr_N\bigl((DW^*E^2WD)^j\bigr)-\E_W\tr_N\bigl((DW^*E^2WD)^j\bigr)\right|>a\right)
  \le2e^{-c_{j,C}a^2N^2}.
  \label{eq:haar-moment-concentration}
\end{align}
Here the $N^2$ exponent results from combining the unitary-group concentration scale with the $N^{-1/2}$ Lipschitz constant in \eqref{eq:haar-moment-lipschitz}; the numerical value of $c_{j,C}>0$ is immaterial.

Apply~\eqref{eq:uniform-haar-moment-centering}--\eqref{eq:haar-moment-concentration} for $1\le j\le K$, choose the centering error smaller than $\eta/2$, and use a union bound. In conjunction with~\eqref{eq:finitely-many-moments-control-levy}, this gives, uniformly over the conditioned singular values,
\begin{align}\label{eq:conditional-product-levy-concentration}
  \PP_{W_N}\left(
  d_{\mathrm L}\bigl(\mu_{H_N},\mu_{A_N^*A_N}\boxtimes\mu_{B_N^*B_N}\bigr)\ge\frac{\varepsilon}{2}\right)
  \le 2K e^{-cN^2},
\end{align}
for all sufficiently large $N$, with $c=c(\varepsilon,C)>0$.

On the event
\begin{align*}
  d_{\mathrm L}(\mu_{A_N^*A_N},\mu_A)\le\delta,\qquad
  d_{\mathrm L}(\mu_{B_N^*B_N},\mu_B)\le\delta
\end{align*}
the triangle inequality,~\eqref{eq:free-product-continuity-choice}, and~\eqref{eq:conditional-product-levy-concentration} imply
\begin{align*}
  d_{\mathrm L}(\mu_{H_N},\mu_A\boxtimes\mu_B)<\varepsilon
\end{align*}
except with conditional probability at most $2Ke^{-cN^2}$.
Consequently,
\begin{align}\label{eq:product-levy-final-bound}
  &\PP\left(
  d_{\mathrm L}\left(\mu_{(B_NA_N)^*(B_NA_N)},\mu_A\boxtimes\mu_B\right)>\varepsilon\right)
  \le
  \PP\left(d_{\mathrm L}(\mu_{A_N^*A_N},\mu_A)>\delta\right)
  +\PP\left(d_{\mathrm L}(\mu_{B_N^*B_N},\mu_B)>\delta\right)
  +2Ke^{-cN^2}.
\end{align}
The first two terms in~\eqref{eq:product-levy-final-bound} are superexponentially negligible at speed $N$ by \eqref{eq:product-input-spectral-concentration}, and the last is also superexponentially negligible at speed $N$. This proves the claim.
The argument for $A_NB_N$ is identical after interchanging $A_N$ and $B_N$; by the moment characterization in Definition~\ref{def:boxtimes} and traciality, $\mu_B\boxtimes\mu_A=\mu_A\boxtimes\mu_B$.
\end{proof}

\subsection{Multiplicativity of the \texorpdfstring{$S$}{S}-transform}

The next theorem gives a radial derivation of the classical multiplicative rule for the $S$-transform \cite[Theorem~2.6]{Voiculescu1987} from~\eqref{eq:finite-radial-product} and the preceding estimates. For deterministic singular values with uniformly bounded norms, its concentration hypotheses are verified in Corollary~\ref{cor:verify-product-hciz}.

\begin{theorem}[multiplicativity of the $S$-transform]
\label{thm:radial-s-transform}
Let $A_N,B_N\in M_N(\C)$ be independent bi-unitarily invariant random matrices with $\norm{A_N},\norm{B_N}\le C$ almost surely for some $C\geq1$.
Assume that the empirical spectral measures of $A_N^*A_N$ and $B_N^*B_N$
satisfy the exponential concentration hypothesis~\eqref{eq:loewner-squeeze-spectral-concentration} of Proposition~\ref{prop:loewner-squeeze}, with respective limits $\mu_A,\mu_B$, where $\mu_A,\mu_B$ are compactly supported probability measures on $[0,\infty)$ with non-zero first moments, and assume that $\mu_B$ is non-degenerate.
Set
\begin{align*}
  \mu_{BA}:=\mu_A\boxtimes\mu_B,
\end{align*}
with $\boxtimes$ defined operator-theoretically in Definition~\ref{def:boxtimes}.
The empirical spectral measure of $(B_NA_N)^*(B_NA_N)$ concentrates around $\mu_{BA}$ as in \eqref{eq:product-output-spectral-concentration}, and
\begin{align}\label{eq:s-transform-multiplicativity}
  S_{\mu_{BA}}(z)=S_{\mu_A}(z)S_{\mu_B}(z)
\end{align}
holds for all sufficiently small complex $z$.
\end{theorem}

\begin{proof}
We derive the $S$-transform of $\mu_{BA}$ from~\eqref{eq:finite-radial-product}; no multiplicativity of the $S$-transform is used as an input.
Write $R_A=R_{\mu_A}$, $R_B=R_{\mu_B}$ and $R_{BA}=R_{\mu_{BA}}$, with analogous shorthand for the $S$-transforms.

By Proposition~\ref{prop:product-levy-concentration}, the empirical spectral measure of $(B_NA_N)^*(B_NA_N)$ concentrates around $\mu_{BA}$ as in~\eqref{eq:product-output-spectral-concentration}. Applying Proposition~\ref{prop:loewner-squeeze} to $A_N^*A_N$, $B_N^*B_N$ and $(B_NA_N)^*(B_NA_N)$, whose norms are bounded by $C^2,C^2,C^4$, respectively, gives
\begin{align}\label{eq:three-radial-r-primitives}
  \Phi_A(t)&=\int_0^t R_A(s)\dd s,& \Phi_B(t)&=\int_0^t R_B(s)\dd s,& \Phi_{BA}(t)&=\int_0^t R_{BA}(s)\dd s
\end{align}
for $|t|<\frac1{64C^4}$.
These limiting transforms are analytic near $0$. Since $\mu_B$ is not a point mass, $\Phi_B''(0)=\Var_{\mu_B}(x)>0$, so $\Phi_B$ is strictly convex near $0$.

Let $x_N,y_N$ be independent Haar unit vectors, independent of $A_N,B_N$, and set $X_N=\norm{A_Nx_N}^2$ and $Y_N=\norm{B_Ny_N}^2$. These variables are independent, and $\norm{B_NA_Nx_N}^2\stackrel{d}{=}X_NY_N$ by~\eqref{eq:finite-radial-product}. Hence
\begin{align*}
  \Phi_{BA}(t)=\lim_{N\to\infty}\frac1N\log\E_{X_N,Y_N}e^{NtX_NY_N}.
\end{align*}

Since $0\le Y_N\le C^2$, the bound $|t|<\frac1{64C^4}$
ensures that $|ty|<1/(64C^2)$ for every $y\in[0,C^2]$.
For each fixed such $t$,
\begin{align*}
  \E_{X_N}e^{NtyX_N}=\exp\{N\Phi_A(ty)+No_N(1)\},
\end{align*}
uniformly for $y\in[0,C^2]$. The uniformity follows because the scaled log-Laplace transforms are convex and converge pointwise on an open interval, hence converge locally uniformly by the remark after Proposition~\ref{prop:local-laplace}.

Choose an open interval $U_B$ about $0$ on which $\Phi_B$ is strictly convex and a sufficiently small compact interval $J\subset\Phi_B'(U_B)$ containing $\bar y=\Phi_B'(0)$ in its interior. The functions $f_t(y)=\Phi_A(ty)$ converge uniformly to $0$ on $[0,C^2]$. Conditioning on $Y_N$, we apply Proposition~\ref{prop:local-laplace} with $Q_N=Y_N$, $L=\Phi_B$ and $f=f_t$; its family clause localizes the global maximizers in $J$, and its perturbation clause absorbs the uniform $o_N(1)$. Consequently, for all sufficiently small real $t$,
\begin{align}\label{eq:radial-product-variational-formula}
  \Phi_{BA}(t)=\sup_{y\in[0,C^2]}\{\Phi_A(ty)-\mathcal I_B(y)\},
\end{align}
where $\mathcal I_B$ is the local Legendre transform of $\Phi_B$ on $U_B$.

After shrinking $J$ and the parameter interval if necessary, the family and analytic clauses of Proposition~\ref{prop:local-laplace} give a unique analytic maximizer $y=y(t)$, with $y(0)=\bar y$. Set $s=ty(t)$ and $u=tR_A(s)$.
The stationarity condition gives $\mathcal I_B'(y(t))=u$. Since $\mathcal I_B'=(\Phi_B')^{-1}$ locally, we obtain $y(t)=\Phi_B'(u)=R_B(u)$, hence
\begin{align}\label{eq:coupled-saddle-equations}
  s=tR_B(u),\qquad u=tR_A(s).
\end{align}
The terms involving $y'(t)$ cancel by the stationary condition when we differentiate~\eqref{eq:radial-product-variational-formula} along $y(t)$. Using~\eqref{eq:three-radial-r-primitives}, we obtain
\begin{align}\label{eq:product-r-transform-saddle}
  R_{BA}(t)=\Phi_{BA}'(t)=y(t)\Phi_A'(ty(t))=R_A(s)R_B(u).
\end{align}
Set $w=tR_{BA}(t)$. By~\eqref{eq:coupled-saddle-equations} and~\eqref{eq:product-r-transform-saddle}, $w=sR_A(s)=uR_B(u)$. Applying~\eqref{eq:s-r-coordinate-relation} to these three identities gives
\begin{align*}
  S_{BA}(w)=\frac1{R_{BA}(t)}=\frac1{R_A(s)R_B(u)}=S_A(w)S_B(w).
\end{align*}
This is exactly~\eqref{eq:s-transform-multiplicativity} for the real values $w=tR_{BA}(t)$ obtained from the saddle-point calculation. Since
\begin{align*}
  \left.\frac{d}{dt}\left(tR_{BA}(t)\right)\right|_{t=0}
  =R_{BA}(0)
  =m_1(\mu_{BA})
  =m_1(\mu_A)m_1(\mu_B)\neq0,
\end{align*}
these values contain a real interval around $0$. Both sides of \eqref{eq:s-transform-multiplicativity} are analytic near $0$, so the identity holds for all sufficiently small complex $z$.
\end{proof}

\begin{remark}
If $\mu_B=\delta_c$ with $c>0$, the same conclusion is immediate: $\mu_A\boxtimes\delta_c$ is the dilation of $\mu_A$ by $c$, and $S_{\delta_c}\equiv 1/c$, so $S_{\mu_A\boxtimes\delta_c}=S_A/c$.
The non-degeneracy assumption in Theorem~\ref{thm:radial-s-transform} is used only to obtain the strict convexity needed in the saddle-point argument.
\end{remark}

\begin{corollary}
\label{cor:derived-s-multiplicativity}
Let $\mu$ and $\nu$ be compactly supported probability measures on $[0,\infty)$ with non-zero first moments. Then
\begin{align}\label{eq:derived-s-multiplicativity}
  S_{\mu\boxtimes\nu}(z)=S_\mu(z)S_\nu(z)
\end{align}
for all sufficiently small complex $z$.
\end{corollary}

\begin{proof}
Choose deterministic nonnegative diagonal matrices $D_N,E_N$, uniformly bounded in dimension, such that the empirical spectral measures of $D_N^*D_N$ and $E_N^*E_N$ converge weakly to $\mu$ and $\nu$, respectively, and attach independent Haar left and right singular-vector matrices.
The concentration hypotheses of Theorem~\ref{thm:radial-s-transform} then hold.
If $\nu$ is not a point mass, that theorem gives~\eqref{eq:derived-s-multiplicativity}.
If $\nu=\delta_c$ with $c>0$, the conclusion is the dilation identity recorded in the point-mass remark following Theorem~\ref{thm:radial-s-transform}.
\end{proof}

\begin{remark}
The $S$-transform identity extends by the identity theorem to any connected common domain of analyticity containing $0$.
\end{remark}

\begin{corollary}
\label{cor:verify-product-hciz}
Let $A_N$ and $B_N$ be independent bi-unitarily invariant matrices whose singular values are deterministic and uniformly bounded.
Assume that
\begin{align*}
  \mu_{A_N^*A_N}\to\mu_A,
  \qquad
  \mu_{B_N^*B_N}\to\mu_B,
\end{align*}
weakly. Then the normalized annealed log-spherical integrals of $A_N$, $B_N$, and $B_NA_N$ satisfy, for $t$ in a neighborhood of $0$,
\begin{align}\label{eq:deterministic-input-radial-limits}
  \Phi_A(t)=\int_0^tR_{\mu_A}(s)\dd s,\qquad
  \Phi_B(t)=\int_0^tR_{\mu_B}(s)\dd s,
\end{align}
and
\begin{align}\label{eq:deterministic-product-radial-limit}
  \Phi_{BA}(t)
  =\int_0^tR_{\mu_A\boxtimes\mu_B}(s)\dd s .
\end{align}
\end{corollary}

\begin{proof}
Since the empirical spectral measures of $A_N^*A_N$ and $B_N^*B_N$ are deterministic and converge weakly,
they satisfy the exponential concentration hypotheses of Proposition~\ref{prop:loewner-squeeze} trivially.
Therefore the annealed limit~\eqref{eq:loewner-squeeze-annealed-limit} gives \eqref{eq:deterministic-input-radial-limits}.
The product concentration~\eqref{eq:product-output-spectral-concentration} holds at the L\'evy-distance scale required by Proposition~\ref{prop:loewner-squeeze}.
Applying \eqref{eq:loewner-squeeze-annealed-limit} to $(B_NA_N)^*(B_NA_N)$ gives~\eqref{eq:deterministic-product-radial-limit}.
\end{proof}

\begin{remark}
The conclusion of Corollary~\ref{cor:verify-product-hciz} remains valid for random singular values, provided the uniform norm bound is retained and the empirical spectral measures of $A_N^*A_N$ and $B_N^*B_N$ satisfy~\eqref{eq:product-input-spectral-concentration}.
Weak convergence in probability alone does not control rare spectral events with probability $e^{-o(N)}$, which may still contribute to the annealed quantity $N^{-1}\log\E e^{N(\cdot)}$.
\end{remark}

\begin{theorem}[finite-dimensional estimate for the $S$-transform]
\label{thm:quantitative-local-multiplicativity}
Fix $0<c<C<\infty$. There is a radius $\rho=\rho(c,C)>0$ such that the following holds.
Let $A_N$ and $B_N$ be independent bi-unitarily invariant $N\times N$ matrices whose singular values are deterministic and contained in $[c,C]$.  Then, for every $\delta>0$, there are $N_0=N_0(\delta,c,C)$ and $\gamma=\gamma(\delta,c,C)>0$ such that, for all $N\ge N_0$,
\begin{align}\label{eq:quantitative-local-s-multiplicativity}
  \mathbb P\left( \sup_{|z|\le\rho}
  \left| S_{\mu_{(B_NA_N)^*(B_NA_N)}}(z) - S_{\mu_{A_N^*A_N}}(z)S_{\mu_{B_N^*B_N}}(z) \right| >\delta
  \right)
  \le e^{-\gamma N^2}.
\end{align}
The same difference tends to $0$ in $L^1$, uniformly over all such deterministic singular-value lists. The conclusions also hold when the singular values are random and belong to $[c,C]$ almost surely.
\end{theorem}

\begin{proof}
Put
$d=\min\{c^2,c^4\}, \ D=\max\{C^2,C^4\}$.
The empirical spectral measures of $A_N^*A_N$ and $B_N^*B_N$ are supported in $[c^2,C^2]$, while both the empirical spectral measure of $(B_NA_N)^*(B_NA_N)$ and $\mu_{A_N^*A_N}\boxtimes\mu_{B_N^*B_N}$ are supported in $[c^4,C^4]$.
Indeed, if $c^2\Id\le a,b\le C^2\Id$, then
$c^4\Id\le a^{1/2}ba^{1/2}\le C^4\Id.$
Thus all measures below are supported in the common interval $[d,D]$.

We first note that there is a disk $|w|\le\rho$, depending only on $d,D$, on which the $S$-transforms of all probability measures supported in $[d,D]$ are defined and depend uniformly continuously on the measure in the weak topology.
Indeed, for such a measure $\theta$, the function
\begin{align*}
  \psi_\theta(z)=\int\frac{zx}{1-zx}\dd\theta(x)
\end{align*}
satisfies $\psi_\theta'(0)=m_1(\theta)\ge d$.
Using the uniform support bound, one may choose $r>0$ such that
\begin{align*}
  |\psi_\theta'(z)-m_1(\theta)|\le\frac d2,\qquad |z|\le r,
\end{align*}
uniformly in $\theta$.
Hence $\psi_\theta$ is uniformly injective on this disk and its inverse $\chi_\theta$ is defined on a common smaller disk.
If $\theta_n\to\theta$ weakly, then
$\psi_{\theta_n}\to\psi_\theta$ uniformly on $|z|\le r$, and therefore $\chi_{\theta_n}\to\chi_\theta$.
Hence
\begin{align*}
  S_{\theta_n}(w)=\frac{1+w}{w}\chi_{\theta_n}(w)
  \longrightarrow S_\theta(w)
\end{align*}
uniformly on a still smaller disk $|w|\le\rho$.
Compactness of the probability measures supported in $[d,D]$ makes this continuity uniform.
Consequently, for every $\delta>0$, there exists $\varepsilon=\varepsilon(\delta,c,C)>0$ such that
\begin{align}
  d_{\mathrm L}(\theta,\theta')\le\varepsilon
  \quad\Longrightarrow\quad
  \sup_{|w|\le\rho}|S_\theta(w)-S_{\theta'}(w)|\le\delta.
  \label{eq:uniform-levy-to-s-continuity}
\end{align}

The conditional estimate established in the proof of Proposition~\ref{prop:product-levy-concentration} is uniform over deterministic singular-value lists bounded by $C$.
Hence, for the $\varepsilon$ above and all sufficiently large $N$,
\begin{align*}
  \mathbb P\left(
  d_{\mathrm L}\left(
  \mu_{(B_NA_N)^*(B_NA_N)},
  \mu_{A_N^*A_N}\boxtimes\mu_{B_N^*B_N}
  \right)>\varepsilon
  \right)
  \le e^{-\gamma N^2},
\end{align*}
for some $\gamma=\gamma(\delta,c,C)>0$.
On the complementary event,~\eqref{eq:uniform-levy-to-s-continuity} and the multiplicativity derived in Corollary~\ref{cor:derived-s-multiplicativity} give the desired comparison.
Indeed, that corollary first gives equality on a neighborhood of $0$; all three $S$-transforms are analytic on the common disk $|w|<\rho$, so the identity theorem extends the equality to the whole disk.
Hence
\begin{align*}
  \sup_{|w|\le\rho}
  \left|
  S_{\mu_{(B_NA_N)^*(B_NA_N)}}(w)-
  S_{\mu_{A_N^*A_N}}(w)S_{\mu_{B_N^*B_N}}(w)
  \right|
  \le\delta.
\end{align*}
This proves~\eqref{eq:quantitative-local-s-multiplicativity}.

The $S$-transforms are uniformly bounded on the common disk, so the same estimate also gives uniform $L^1$ convergence.
If the singular values are random but lie in $[c,C]$ almost surely, conditioning on them gives the same estimate with constants depending only on $c,C,\delta$, and hence the unconditional conclusion.
\end{proof}

\section{From tensor products to products of matrices}
\label{sec:two-leg}

We use the vectorization convention
\begin{align*}
  (A\otimes B)\operatorname{vec}(M)
  = \operatorname{vec}(AMB^{\mathsf T}).
\end{align*}

\begin{proposition}
\label{cor:exact-spherical-two-leg}
Let $X_1$ be an arbitrary random matrix and let $X_2$ be a bi-unitarily invariant random matrix, independent of $X_1$. Let $\xi$ be Haar distributed on the unit sphere of $\C^N\otimes\C^N$, independent of $X_1,X_2$.
Then
\begin{align}\label{eq:two-leg-spherical-reduction}
  \norm{(X_1\otimes X_2)\xi}^2
  \stackrel{d}{=}
  \norm{((X_1X_2)\otimes\Id_N)\xi}^2.
\end{align}
\end{proposition}

\begin{proof}
Write $\xi=\operatorname{vec}(M)$. Since vectorization is an isometry, $M$ is uniform on the Hilbert--Schmidt unit sphere and is bi-unitarily invariant.

For independent bi-unitarily invariant matrices $X,Y$, we have $XY\stackrel{d}{=}YX$; see \cite[Section~III.B]{IpsenKieburg2014} for the case of laws with densities.
Indeed, conditional on their singular values, the singular values of $XY$ have the law of those of $D_1WD_2$, where $W$ is Haar. Transposition gives the same law as for $D_2WD_1$, hence for $YX$. Both products are bi-unitarily invariant, so their matrix laws agree.

Apply this to $M$ and $X_2$. Since $X_2^{\mathsf T}\stackrel{d}{=}X_2$ and $X_1$ is independent of $(M,X_2)$, the vectorization convention gives
\begin{align*}
  \norm{(X_1\otimes X_2)\xi}^2
  =\norm{X_1MX_2^{\mathsf T}}_{\mathrm{HS}}^2
  \stackrel{d}{=}\norm{X_1X_2M}_{\mathrm{HS}}^2
  =\norm{((X_1X_2)\otimes\Id_N)\xi}^2.
\end{align*}
\end{proof}

\begin{theorem}
\label{thm:exact-two-leg-hciz}
Let $\Pi_N$ be a deterministic rank-one projection on $\C^N\otimes\C^N$.
Let $U_N$ be Haar distributed on $\U(N^2)$, independent of the matrix factors.
Let $X_{1,N}\in M_N(\C)$ be arbitrary and let $X_{2,N}\in M_N(\C)$ be bi-unitarily invariant, independent of $X_{1,N}$. Set
\begin{align*}
  \mathcal X_N=X_{1,N}\otimes X_{2,N},
  \qquad \mathcal P_N=(X_{1,N}X_{2,N})\otimes \Id_N .
\end{align*}
Then, for every $N$ and every real parameter $t$, the following identity holds in $[0,+\infty]$:
\begin{equation}\label{eq:exact-two-leg-hciz}
  \begin{aligned}
    &\E_{X_1,X_2}
    \int_{\U(N^2)}
    \exp\Bigl\{
    N^2 t\,\Tr\bigl( \mathcal X_N^*\mathcal X_N\,U_N\Pi_NU_N^* \bigr) \Bigr\}\,\dd U_N
    \\ &\qquad =
    \E_{X_1,X_2}
    \int_{\U(N^2)}
    \exp\Bigl\{ N^2 t\,\Tr\bigl( \mathcal P_N^*\mathcal P_N\,U_N\Pi_NU_N^* \bigr) \Bigr\}\,\dd U_N .
  \end{aligned}
\end{equation}
In terms of the rank-one spherical integral, this identity reads
\begin{align*}
  \E_{X_1,X_2}I_{N^2}(t;\mathcal X_N^*\mathcal X_N)=\E_{X_1,X_2}I_{N^2}(t;\mathcal P_N^*\mathcal P_N).
\end{align*}
\end{theorem}

\begin{proof}
The projection $U_N\Pi_NU_N^*$ is the projection onto a Haar unit vector $\xi_N\in\C^N\otimes\C^N$.
The left-hand side of \eqref{eq:exact-two-leg-hciz} is
\begin{align*}
  \E_{X_1,X_2,\xi_N}
  \exp\{N^2 t\,\norm{(X_{1,N}\otimes X_{2,N})\xi_N}^2\},
\end{align*}
and the right-hand side is the same expression with $(X_{1,N}X_{2,N})\otimes\Id_N$ in place of $X_{1,N}\otimes X_{2,N}$.
The equality in distribution of the two squared norms is Proposition~\ref{cor:exact-spherical-two-leg}.
Therefore their exponential moments are equal by \eqref{eq:two-leg-spherical-reduction}, proving \eqref{eq:exact-two-leg-hciz}.
\end{proof}

\begin{remark}
The equality in Theorem~\ref{thm:exact-two-leg-hciz} is annealed over $X_{2,N}$, and also over $X_{1,N}$ when the latter is random; only $X_{2,N}$ is required to be bi-unitarily invariant.
For fixed matrix realizations, the spectra of $(X_{1,N}\otimes X_{2,N})^*(X_{1,N}\otimes X_{2,N})$ and $((X_{1,N}X_{2,N})\otimes\Id_N)^*((X_{1,N}X_{2,N})\otimes\Id_N)$ need not coincide.
The exact statement is therefore a distributional identity after the stated averaging, not a quenched identity for each realization. 
Corollary~\ref{prop:higher-leg-counterexample} shows that, for every $k\ge3$, even an asymptotic analogue fails in general for sufficiently small positive parameters.
\end{remark}

\section{Global Haar and product-state vectors for tensors}
\label{sec:tensor-regimes}

There are two natural choices of test vector for a tensor product.
A global Haar vector in the full tensor-product space sees only the empirical spectral measure of $(X_{1,N}\otimes\cdots\otimes X_{k,N})^* (X_{1,N}\otimes\cdots\otimes X_{k,N})$ and therefore gives classical multiplication of eigenvalues.
A product-state vector keeps the one-leg speed $N$ and gives the free multiplicative convolution under the hypotheses of Theorem~\ref{thm:local-free}.
Both statements hold for any fixed number of legs. The equality in Theorem~\ref{thm:exact-two-leg-hciz} is specific to two legs.
The global statement is purely spectral after conditioning and requires neither invariance nor independence among the matrix factors.
The product-state statement uses independence of the variables $\norm{X_{i,N}x_{i,N}}^2$.

\subsection{Global Haar vectors}

\begin{theorem}[global Haar vector and classical convolution]
\label{thm:global-classical}
Fix $k\ge2$.
For each $i=1,\ldots,k$, let
$X_{i,N}\in M_N(\C)$ be an arbitrary random matrix; the matrices may
be mutually dependent.  Let $\mu_i$ be a compactly supported
probability measure on $[0,\infty)$.  Assume that
\begin{align}\label{eq:global-factor-gram-convergence}
  \mu_{X_{i,N}^*X_{i,N}}\longrightarrow \mu_i
\end{align}
weakly in probability, and that there is $0<M<\infty$ such that
\begin{align*}
  \mathbb P\left(\max_{1\le i\le k}\norm{X_{i,N}}\le M\right)\longrightarrow1.
\end{align*}
No independence or unitary-invariance assumption is
made on the joint matrix family.  Set
\begin{align*}
  \mathcal X_N=X_{1,N}\otimes\cdots\otimes X_{k,N},
  \qquad
  \nu_{\mathrm{cl}}=\mu_1\circledast\cdots\circledast\mu_k .
\end{align*}
Let $\xi_N$ be Haar distributed on the unit sphere of $(\C^N)^{\otimes k}$, independent of the joint matrix family. Here $\E_{\xi_N}$ denotes expectation only over the Haar vector, conditional on the matrices.
Then, for every fixed real $t$ with $|t|<1/(64M^{2k})$,
\begin{align}\label{eq:global-conditional-hciz-limit}
  \frac1{N^k}\log I_{N^k}(t;\mathcal X_N^*\mathcal X_N)
  =\frac1{N^k}\log \E_{\xi_N}\exp\{N^k t\,\norm{\mathcal X_N\xi_N}^2\}
  \xrightarrow[N\to\infty]{}\int_0^tR_{\nu_{\mathrm{cl}}}(s)\dd s,
\end{align}
in probability.
If the singular values of the $X_{i,N}$ are deterministic, the conditional integral depends only on these singular values and is therefore deterministic, even if the singular vectors are random.
In this case, for every $N$,
\begin{align*}
  J_{N^k}(t;\mathcal X_N)=\E_{X_{1,N},\ldots,X_{k,N}}I_{N^k}(t;\mathcal X_N^*\mathcal X_N)=I_{N^k}(t;\mathcal X_N^*\mathcal X_N).
\end{align*}
Using the rank-one HCIZ representation in Remark~\ref{rem:rank-one-hciz-representation}, we obtain the annealed limit
\begin{align}\label{eq:global-deterministic-hciz-limit}
  \lim_{N\to\infty}\frac1{N^k}\log\E_{X_{1,N},\ldots,X_{k,N}}I_{N^k}(t;\mathcal X_N^*\mathcal X_N)
  &=\lim_{N\to\infty}\frac1{N^k}\log \int_{\U(N^k)}\exp\Bigl\{ N^k t\,\Tr\bigl( \mathcal X_N^*\mathcal X_N\,U\Pi_NU^*\bigr) \Bigr\}\dd U \notag \\
  &=\int_0^t R_{\nu_{\mathrm{cl}}}(s)\dd s,
\end{align}
for every deterministic rank-one projection $\Pi_N$.

If, in addition, the smallest and largest eigenvalues of $X_{i,N}^*X_{i,N}$ converge in probability to the corresponding edges of $\supp\mu_i$ for every $i$, then~\eqref{eq:global-conditional-hciz-limit} holds for every fixed real $t\in\mathcal D_{\nu_{\mathrm{cl}}}$ as defined in~\eqref{eq:real-r-transform-domain}. For deterministic singular values,~\eqref{eq:global-deterministic-hciz-limit} holds on this domain as well.
\end{theorem}

\begin{proof}
Condition on the matrices.
The argument below uses only the eigenvalues of $X_{i,N}^*X_{i,N}$. Since we have the factorization
\begin{align*}
  \mathcal X_N^*\mathcal X_N
  =X_{1,N}^*X_{1,N}\otimes\cdots\otimes X_{k,N}^*X_{k,N},
\end{align*}
the eigenvalues of $\mathcal X_N^*\mathcal X_N$ are exactly the products $\lambda_{1,j_1}\cdots\lambda_{k,j_k}$, where $\lambda_{i,j}$ are the eigenvalues of $X_{i,N}^*X_{i,N}$.
Hence the empirical spectral measure of $\mathcal X_N^*\mathcal X_N$ is
\begin{align*}
  \mu_{X_{1,N}^*X_{1,N}}\circledast\cdots\circledast
  \mu_{X_{k,N}^*X_{k,N}},
\end{align*}
which converges weakly in probability to $\nu_{\mathrm{cl}}$.
From any subsequence, extract a further subsequence along which the empirical spectral measure of $\mathcal X_N^*\mathcal X_N$ converges weakly almost surely and $\max_i\norm{X_{i,N}}\le M$ eventually almost surely. For the latter assertion, choose the failure probabilities to be summable and apply the Borel--Cantelli lemma. Along almost every such realization, $0\le\mathcal X_N^*\mathcal X_N\le M^{2k}\Id_{N^k}$ eventually, so \cite[Theorem~2]{GuionnetMaida2005}, with $\beta=2$, gives~\eqref{eq:global-conditional-hciz-limit} in dimension $N^k$ for each fixed real $|t|<1/(64M^{2k})$. The subsequence criterion therefore gives convergence in probability.

Under the additional edge assumptions, the extreme eigenvalues of $\mathcal X_N^*\mathcal X_N$ converge in probability to the corresponding edges of $\supp\nu_{\mathrm{cl}}$, since they are products of the corresponding factor eigenvalues. The same subsequence argument, now also taking almost sure limits of the extreme eigenvalues and applying Theorem~\ref{thm:rank-one-hciz}, proves~\eqref{eq:global-conditional-hciz-limit} for every fixed real $t\in\mathcal D_{\nu_{\mathrm{cl}}}$.
For deterministic singular values, $I_{N^k}(t;\mathcal X_N^*\mathcal X_N)$ is deterministic, so averaging over the matrices leaves its value unchanged.
\end{proof}

\begin{remark}[quenched versus annealed]
The limit~\eqref{eq:global-conditional-hciz-limit} is quenched with respect to the matrix factors.
Taking the matrix expectation outside the logarithm gives the averaged quenched quantity, while placing it inside the logarithm gives the annealed quantity. Convergence in probability of the conditional normalized logarithm alone does not imply the annealed limit.
Rare matrix events can contribute at the $N^k$ exponential scale, so such an annealed conclusion requires additional tail or concentration assumptions.
\end{remark}

\begin{corollary}[product of two matrices at speed \texorpdfstring{$N^2$}{N2}]
\label{cor:two-leg-product-classical}
Let $X_{1,N}$ be arbitrary and let $X_{2,N}$ be independent of $X_{1,N}$ and bi-unitarily invariant.
Suppose that both matrices have deterministic, uniformly bounded singular values and that $\mu_{X_{i,N}^*X_{i,N}}\to\mu_i$ weakly for $i=1,2$, where $\mu_1,\mu_2$ are compactly supported probability measures on $[0,\infty)$.
Let
\begin{align}\label{eq:two-leg-product-amplification}
  \mathcal P_N=(X_{1,N}X_{2,N})\otimes\Id_N
\end{align}
on $\C^N\otimes\C^N$, let $\Pi_N$ be a deterministic rank-one projection, and set $\nu_{\mathrm{cl}}=\mu_1\circledast\mu_2$.
Then
for all real $t$ sufficiently close to $0$,
\begin{align}\label{eq:two-leg-product-speed-n2-limit}
  \lim_{N\to\infty}\frac1{N^2}\log\E_{X_1,X_2}I_{N^2}(t;\mathcal P_N^*\mathcal P_N)
  &=\lim_{N\to\infty}\frac1{N^2}\log\E_{X_1,X_2}\int_{\U(N^2)}\exp\left\{N^2t\Tr\left(\mathcal P_N^*\mathcal P_N U\Pi_NU^*\right)\right\}\dd U\notag\\
  &=\int_0^tR_{\nu_{\mathrm{cl}}}(s)\dd s.
\end{align}

If, in addition, the smallest and largest eigenvalues of $X_{i,N}^*X_{i,N}$ converge to the corresponding edges of $\supp\mu_i$ for $i=1,2$, then~\eqref{eq:two-leg-product-speed-n2-limit} holds for every real $t\in\mathcal D_{\nu_{\mathrm{cl}}}$.
\end{corollary}

\begin{proof}
By Theorem~\ref{thm:exact-two-leg-hciz} and the definition of $J$, for every $N$ and every real $t$,
\begin{align*}
  \E_{X_1,X_2}I_{N^2}(t;\mathcal P_N^*\mathcal P_N)=J_{N^2}(t;X_{1,N}\otimes X_{2,N}).
\end{align*}
Theorem~\ref{thm:global-classical} now gives the limit near $0$ and, under the additional edge assumptions, on $\mathcal D_{\nu_{\mathrm{cl}}}$.
\end{proof}

\begin{remark}[second cumulant check]
For a fixed Hermitian matrix $H$ in dimension $D$ and a Haar unit vector $\xi\in\C^D$,
\begin{align*}
  \Var_\xi(\xi^*H\xi)
  = \frac{\Tr(H^2)-D^{-1}\Tr(H)^2}{D(D+1)}
  = \frac{\Var(\mu_H)}{D+1}.
\end{align*}
Thus,
\begin{align*}
  \left.\frac{d^2}{dt^2}\right|_{t=0}
  \frac1D\log\E_\xi e^{Dt\,\xi^*H\xi}
  = \frac{D}{D+1}\Var(\mu_H).
\end{align*}
For compactly supported probability measures $\mu_1,\mu_2$ on $[0,\infty)$, independence and the non-crossing moment--cumulant formula give, respectively,
\begin{align*}
  \Var(\mu_1\circledast\mu_2)&=\Var(\mu_1)m_1(\mu_2)^2+m_1(\mu_1)^2\Var(\mu_2)+\Var(\mu_1)\Var(\mu_2),\\
  \Var(\mu_1\boxtimes\mu_2)&=\Var(\mu_1)m_1(\mu_2)^2+m_1(\mu_1)^2\Var(\mu_2).
\end{align*}
Hence
\begin{align}\label{eq:classical-free-variance-gap}
  \Var(\mu_1\circledast\mu_2)-\Var(\mu_1\boxtimes\mu_2)
  =\Var(\mu_1)\Var(\mu_2).
\end{align}
Equation~\eqref{eq:classical-free-variance-gap} shows that the classical and free multiplicative convolutions already differ at second order unless at least one factor has zero variance.
\end{remark}

\begin{corollary}[averaging before and after the logarithm]
\label{cor:two-orders-of-averaging}
Let $A_N,B_N$ be independent bi-unitarily invariant matrices with deterministic, uniformly bounded singular values.
Suppose that the empirical spectral measures of $A_N^*A_N$ and $B_N^*B_N$ converge weakly to compactly supported probability measures $\mu_A,\mu_B$ on $[0,\infty)$.
Set $H_N=(B_NA_N)^*(B_NA_N)$. For all real $t$ sufficiently close to $0$,
\begin{align*}
  \lim_{N\to\infty}\frac1{N^2}\E_{A_N,B_N}\left[\log I_{N^2}(t;H_N\otimes\Id_N)\right]&=\int_0^tR_{\mu_A\boxtimes\mu_B}(s)\dd s,\\
  \lim_{N\to\infty}\frac1{N^2}\log\E_{A_N,B_N}\left[I_{N^2}(t;H_N\otimes\Id_N)\right]&=\int_0^tR_{\mu_A\circledast\mu_B}(s)\dd s.
\end{align*}
The first line gives the limit of the averaged quenched normalized logarithm, and the second gives the annealed limit.
The difference between these limits satisfies
\begin{align}
  \label{eq:conclusion-averaging-gap}
  \int_0^t\left(R_{\mu_A\circledast\mu_B}(s)-R_{\mu_A\boxtimes\mu_B}(s)\right)\dd s
  =\frac{t^2}{2}\Var(\mu_A)\Var(\mu_B)+O(t^3).
\end{align}
\end{corollary}

\begin{proof}
Proposition~\ref{prop:product-levy-concentration} gives $\mu_{H_N}\to\mu_A\boxtimes\mu_B$ weakly in probability. If $\norm{A_N},\norm{B_N}\le C$ with $C\ge1$, then $0\le H_N\otimes\Id_N\le C^4\Id_{N^2}$ and $\mu_{H_N\otimes\Id_N}=\mu_{H_N}$. The small-parameter condition in
\cite[Theorem~2]{GuionnetMaida2005}
holds uniformly for $|t|<1/(64C^4)$. Applying that theorem in dimension $N^2$ along almost surely convergent subsequences of the empirical measures gives
\begin{align*}
  \frac1{N^2}\log I_{N^2}(t;H_N\otimes\Id_N)
  \longrightarrow\int_0^tR_{\mu_A\boxtimes\mu_B}(s)\dd s
\end{align*}
in probability.
The bound $\left|N^{-2}\log I_{N^2}(t;H_N\otimes\Id_N)\right|\le|t|C^4$ also gives convergence in $L^1$, so taking expectations proves the first limit.
The second follows from Corollary~\ref{cor:two-leg-product-classical}, with $X_{1,N}=B_N$ and $X_{2,N}=A_N$.
Finally, $R_\mu(s)=m_1(\mu)+\Var(\mu)s+O(s^2)$ and~\eqref{eq:classical-free-variance-gap} give~\eqref{eq:conclusion-averaging-gap}.
\end{proof}

For comparison, consider the same operator $H_N\otimes\Id_N$ at speed $N$ with the product-state vector $x_N\otimes y_N$, where $x_N,y_N$ are independent Haar unit vectors in $\C^N$, independent of the matrices.
For fixed $A_N,B_N$, we have $(x_N\otimes y_N)^*(H_N\otimes\Id_N)(x_N\otimes y_N)=x_N^*H_Nx_N$, so the integral is $I_N(t;H_N)$.
Under the hypotheses of the corollary, both orders of matrix averaging have the same limit $\int_0^tR_{\mu_A\boxtimes\mu_B}(s)\dd s$ for sufficiently small real $t$, by the preceding argument in dimension $N$ and Propositions~\ref{prop:product-levy-concentration} and~\ref{prop:loewner-squeeze}.

\subsection{Product-state vectors}

The following consequence of Proposition~\ref{prop:local-laplace} and Corollary~\ref{cor:derived-s-multiplicativity} will be used in the proof of Theorem~\ref{thm:local-free}.

\begin{proposition}
\label{prop:scalar-product-step}
Let $P_N,Q_N$ be independent nonnegative random variables, uniformly bounded in $N$.
Let $\nu$ and $\mu$ be compactly supported probability measures on $[0,\infty)$ with non-zero first moments.
Assume that, for $s$ near $0$,
\begin{align}\label{eq:scalar-product-input}
  \lim_{N\to\infty}\frac1N\log\E e^{NsP_N}
  =\int_0^s R_\nu(u)\dd u,
  \qquad \lim_{N\to\infty}\frac1N\log\E e^{NsQ_N}
  =\int_0^s R_\mu(u)\dd u,
\end{align}
and that $\mu$ is non-degenerate. Then, for $t$ near $0$,
\begin{align}\label{eq:scalar-product-output-log-laplace}
  \lim_{N\to\infty}\frac1N\log\E e^{NtP_NQ_N}
  =\int_0^t R_{\nu\boxtimes\mu}(u)\dd u .
\end{align}
\end{proposition}

\begin{proof}
Set $\Phi_P(s)=\int_0^sR_\nu(u)\dd u$ and $\Phi_Q(s)=\int_0^sR_\mu(u)\dd u$, and choose $C_Q>0$ with $0\le Q_N\le C_Q$. By convexity, the convergences in~\eqref{eq:scalar-product-input} are locally uniform for real $s$ near $0$. Independence therefore gives, for each sufficiently small real $t$,
\begin{align*}
  \E\left[e^{NtP_NQ_N}\mid Q_N\right]=\exp\{N\Phi_P(tQ_N)+N\varepsilon_N(t,Q_N)\},
  \qquad \sup_{q\in[0,C_Q]}|\varepsilon_N(t,q)|\to0.
\end{align*}

Since $\Phi_Q''(0)=\Var(\mu)>0$, choose a small open interval $U_Q$ about $0$ on which $\Phi_Q$ is strictly convex, and a compact interval $J\subset\Phi_Q'(U_Q)$ containing $\Phi_Q'(0)$ in its interior. The functions $q\mapsto\Phi_P(tq)$ converge uniformly to $0$ on $[0,C_Q]$. The localization and perturbation clauses of Proposition~\ref{prop:local-laplace} thus give
\begin{align}\label{eq:scalar-product-variational-formula}
  \Phi_{PQ}(t):=\lim_{N\to\infty}\frac1N\log\E e^{NtP_NQ_N}
  =\max_{q\in[0,C_Q]}\{\Phi_P(tq)-\mathcal I_Q(q)\},
\end{align}
where $\mathcal I_Q(q)=\sup_{s\in U_Q}\{sq-\Phi_Q(s)\}$. By the analytic clause of the same proposition, after shrinking the parameter interval, the maximizer $q(t)$ is unique and analytic, with $q(0)=m_1(\mu)$.

Set $s=tq(t)$ and $u=tR_\nu(s)$.
The stationarity condition gives $\mathcal I_Q'(q(t))=u$. Since $\mathcal I_Q'=(\Phi_Q')^{-1}$ locally, we obtain
$q(t)=R_\mu(u)$. Differentiating the value at the maximizer gives
\begin{align*}
  s=tR_\mu(u),\qquad u=tR_\nu(s),\qquad
  y(t):=\Phi_{PQ}'(t)=R_\mu(u)R_\nu(s).
\end{align*}
Thus $ty(t)=sR_\nu(s)=uR_\mu(u)$, and~\eqref{eq:s-r-coordinate-relation} together with Corollary~\ref{cor:derived-s-multiplicativity} yields
\begin{align*}
  y(t)S_{\nu\boxtimes\mu}(ty(t))=1.
\end{align*}
The function $R_{\nu\boxtimes\mu}(t)$ satisfies the same equation and has the same value $m_1(\nu)m_1(\mu)$ at $0$. The implicit-function theorem gives local uniqueness, since the derivative of the left-hand side with respect to $y$ at $t=0$ is $S_{\nu\boxtimes\mu}(0)=1/(m_1(\nu)m_1(\mu))\ne0$. Hence $\Phi_{PQ}'(t)=R_{\nu\boxtimes\mu}(t)$ near $0$. Integrating from $0$ proves~\eqref{eq:scalar-product-output-log-laplace}.
\end{proof}

For matrices with deterministic singular values as in Corollary~\ref{cor:local-free-deterministic},
the fixed-order cumulant asymptotics of the product-state integral follow from \cite[Lemma~5.3]{CollinsGurauLionniII}.
If these deterministic singular values are uniformly bounded away from zero, the limit in~\eqref{eq:local-product-state-log-laplace} can also be deduced from the rank-one asymptotics in \cite{Husson2025}, using independence, the multiplicativity of the $S$-transform and a local form of the Laplace principle.
The following theorem establishes convergence of the normalized logarithm for sufficiently small real parameters under assumptions on the individual annealed rank-one spherical integrals.

\begin{theorem}[product-state vector and free convolution]
\label{thm:local-free}
Fix $k\ge2$.
Let $X_{1,N},\ldots,X_{k,N}$ be independent random matrices.
Assume that there is $M<\infty$ such that $\norm{X_{i,N}}\leq M$ almost surely for every $i,N$.
Let $x_{1,N},\ldots,x_{k,N}$ be independent Haar unit vectors in $\C^N$, independent of the matrices, and set
\begin{align}\label{eq:local-product-state-variable}
  Z_N=
  \norm{(X_{1,N}\otimes\cdots\otimes X_{k,N})(x_{1,N}\otimes\cdots\otimes x_{k,N})}^2 .
\end{align}
Let $\mu_1$ be a compactly supported probability measure on $[0,\infty)$ with non-zero first moment, and let $\mu_2,\ldots,\mu_k$ be compactly supported non-degenerate probability measures on $[0,\infty)$.
Assume that, for each $i=1,\ldots,k$, the normalized logarithms of the annealed rank-one spherical integrals satisfy, for $t$ in a neighborhood of $0$,
\begin{align}\label{eq:one-leg-log-spherical-assumption}
  \Phi_i(t)
  :=\lim_{N\to\infty}\frac1N\log\E_{X_{i,N},x_{i,N}}\exp\{Nt\norm{X_{i,N}x_{i,N}}^2\}=\int_0^tR_{\mu_i}(s)\dd s.
\end{align}
For $i\geq2$, non-degeneracy and non-negativity imply $m_1(\mu_i)>0$, so every $S_{\mu_i}$ below is well defined.

Let $\eta=\mu_1\boxtimes\cdots\boxtimes\mu_k$, with $\boxtimes$ as in Definition~\ref{def:boxtimes}. Then, for real $t$ sufficiently close to $0$,
\begin{align}\label{eq:local-product-state-log-laplace}
  \lim_{N\to\infty}\frac1N\log\E\exp\{NtZ_N\}
  =\int_0^t R_\eta(s)\dd s .
\end{align}
By Corollary~\ref{cor:derived-s-multiplicativity},
\begin{align}\label{eq:local-product-state-s-transform}
  S_\eta(z)=S_{\mu_1}(z)\cdots S_{\mu_k}(z)
\end{align}
for all sufficiently small complex $z$.
\end{theorem}

\begin{proof}
Put
\begin{align}\label{eq:one-leg-radial-factors}
  Q_{i,N}:=\norm{X_{i,N}x_{i,N}}^2.
\end{align}
Since the norm is multiplicative on simple tensors, for every realization,
\begin{align*}
  Z_N=\prod_{i=1}^kQ_{i,N},
\end{align*}
and the $k$ factors defined in~\eqref{eq:one-leg-radial-factors} are independent.

For $i\ge2$, $\Phi_i''(0)=\Var(\mu_i)>0$, so $\Phi_i$ is strictly convex near $0$. For $1\le j\le k$, set $P_N^{(j)}=\prod_{i=1}^jQ_{i,N}$ and $\eta_j=\mu_1\boxtimes\cdots\boxtimes\mu_j$. The variables $P_N^{(j)}$ are bounded by $M^{2j}$, and the measures $\eta_j$ are compactly supported with $m_1(\eta_j)=\prod_{i=1}^jm_1(\mu_i)>0$.

We prove inductively that
\begin{align}\label{eq:iterated-product-log-laplace}
  \lim_{N\to\infty}\frac1N\log\E e^{NtP_N^{(j)}}=\int_0^tR_{\eta_j}(s)\dd s
\end{align}
for real $t$ sufficiently close to $0$. The case $j=1$ is~\eqref{eq:one-leg-log-spherical-assumption}. If the formula holds for $j<k$, then $P_N^{(j)}$ and $Q_{j+1,N}$ are independent, and Proposition~\ref{prop:scalar-product-step} gives the formula for $j+1$, since $\eta_{j+1}=\eta_j\boxtimes\mu_{j+1}$. The neighborhood may shrink at each step and thus depend on the fixed number $k$, but not on $N$. Taking $j=k$ proves~\eqref{eq:local-product-state-log-laplace}; Corollary~\ref{cor:derived-s-multiplicativity} gives~\eqref{eq:local-product-state-s-transform}.
\end{proof}

The measure $\eta$ is the free-convolution measure defined in Definition~\ref{def:boxtimes}. To distinguish it from the ordinary weak limit, note that~\eqref{eq:one-leg-log-spherical-assumption}, the expansion $\Phi_i(t)=m_1(\mu_i)t+O(t^2)$, and exponential Chebyshev bounds give $Q_{i,N}\to m_1(\mu_i)$ in probability. Consequently, for each $j$, the law of $P_N^{(j)}$ converges weakly to $\delta_{\prod_{i=1}^jm_1(\mu_i)}$.

\begin{corollary}
Under the hypotheses of Theorem~\ref{thm:local-free}, assume in
addition that the matrices $X_{i,N}$ are right-unitarily invariant.
If $x_N$ is a single independent Haar unit vector, then the annealed limit \eqref{eq:local-product-state-log-laplace} and the $S$-transform identity \eqref{eq:local-product-state-s-transform} remain valid with $x_N^{\otimes k}$ in place of $x_{1,N}\otimes\cdots\otimes x_{k,N}$.

If, in addition, the singular values of each $X_{i,N}$ are deterministic and $\mu_{X_{i,N}^*X_{i,N}}\to\mu_i$ weakly, then the following quenched convergence holds for every fixed real $t$ in a possibly smaller neighborhood of $0$:
\begin{align*}
\frac1N\log\E_{x_N}\exp\left\{Nt\prod_{i=1}^k\norm{X_{i,N}x_N}^2\right\}
\longrightarrow\int_0^tR_{\mu_1\boxtimes\cdots\boxtimes\mu_k}(s)\dd s
\end{align*}
in probability and in $L^1$ with respect to the matrices. Here $\E_{x_N}$ denotes expectation over $x_N$ with the matrices held fixed.
\end{corollary}

\begin{proof}
Conditional on $x_N$, the variables $\norm{X_{i,N}x_N}^2$ are independent, and right-unitary invariance makes each conditional law independent of $x_N$. 
They therefore have the same joint law as the independent factors used in Theorem~\ref{thm:local-free}.

For the quenched statement, set $Q_{i,N}=\norm{X_{i,N}x_N}^2$, $Q_N=(Q_{1,N},\ldots,Q_{k,N})$ and $F_N(t)=N^{-1}\log\E_{x_N}e^{Nt\prod_iQ_{i,N}}$.
Write $\eta=\mu_1\boxtimes\cdots\boxtimes\mu_k$ and $\Phi(t)=\int_0^tR_\eta(u)\dd u$.
The matrices $X_{1,N}^*X_{1,N},\ldots,X_{k,N}^*X_{k,N}$ are independent and unitarily invariant, with deterministic spectra.
Their asymptotic freeness and concentration in \cite{Collins2003}, together with \cite[Theorem~2]{GuionnetMaida2005} and additivity of the $R$-transform, give
\begin{align*}
\Lambda_N(s):=\frac1N\log\E_{x_N}e^{Ns\cdot Q_N}
\longrightarrow\Lambda(s):=\sum_{i=1}^k\Phi_i(s_i)
\end{align*}
in probability for $s\in\R^k$ sufficiently close to $0$.
Indeed, $\E_{x_N}e^{Ns\cdot Q_N}=I_N(1;\sum_i s_iX_{i,N}^*X_{i,N})$. The bound $\norm{\sum_i s_iX_{i,N}^*X_{i,N}}\le M^2\sum_i|s_i|$ allows us to apply the rank-one asymptotic formula for $s$ sufficiently close to $0$.

For small $t$, set $w=tR_\eta(t)$, $q_i=S_{\mu_i}(w)^{-1}$ and $s_i=wS_{\mu_i}(w)$.
By \eqref{eq:s-r-coordinate-relation} and Corollary~\ref{cor:derived-s-multiplicativity},
\begin{align*}
q_i=R_{\mu_i}(s_i)=\partial_i\Lambda(s),\qquad
\prod_{i=1}^kq_i=R_\eta(t),\qquad
s_i=t\prod_{j\ne i}q_j.
\end{align*}
The functions $w$, $q_i$ and $s_i$ are analytic in $t$ near $0$.

For each $N$ and fixed matrices, let $\widetilde{\mathbb P}_s$ be the probability measure with density $e^{N(s\cdot Q_N-\Lambda_N(s))}$ relative to the law of $x_N$, and write $\widetilde{\mathbb E}_s$ for expectation under this measure.
Applying Markov's inequality to $e^{\pm NhQ_{i,N}}$, as in the proof of Proposition~\ref{prop:local-laplace}, and using differentiability of $\Lambda$ and convergence of $\Lambda_N(s\pm he_i)$, shows that $\widetilde{\mathbb P}_s(\norm{Q_N-q}>\varepsilon)\to0$ in probability with respect to the matrices for every $\varepsilon>0$.
Since $Q_N$ is uniformly bounded, Jensen's inequality gives
\begin{align*}
F_N(t)\ge\Lambda_N(s)+\widetilde{\mathbb E}_s\left[t\prod_{i=1}^kQ_{i,N}-s\cdot Q_N\right],
\end{align*}
whose right-hand side converges in probability to $\Lambda(s)-s\cdot q+t\prod_iq_i$.
Since $\nabla\Lambda(s)=q$ and $s_i=t\prod_{j\ne i}q_j$, the function $t\mapsto\Lambda(s)-s\cdot q+t\prod_iq_i$ has derivative $\prod_iq_i=R_\eta(t)$.
It vanishes at $t=0$, so it equals $\Phi(t)$.

For the upper bound, the annealed conclusion already proved and Markov's inequality yield, for every $\varepsilon>0$,
\begin{align*}
\PP(F_N(t)>\Phi(t)+\varepsilon)
\le e^{-N(\Phi(t)+\varepsilon)}\E e^{NF_N(t)}
=e^{-N\varepsilon+o(N)}.
\end{align*}
Thus $F_N(t)\to\Phi(t)$ in probability. The bound $|F_N(t)|\le|t|M^{2k}$ gives convergence in $L^1$.
\end{proof}

\begin{remark}[point-mass factors]
The non-degeneracy assumption is not needed for point-mass factors. If $\mu_i=\delta_{a_i}$, then the one-leg assumption gives
\begin{align*}
  \Lambda_{i,N}(s):=\frac1N\log\E e^{NsQ_{i,N}}\longrightarrow a_is
\end{align*}
locally uniformly near $0$. Hence, for any uniformly bounded product $P_N$ of the remaining independent factors,
\begin{align*}
  \frac1N\log\E e^{NtP_NQ_{i,N}}=\frac1N\log\E e^{Na_itP_N}+o(1)
\end{align*}
uniformly for $t$ in a sufficiently small neighborhood of $0$. Thus each point-mass factor acts as the rescaling $t\mapsto a_it$, and the argument extends to any number of such factors. For $a_i>0$, this is consistent with $S_{\delta_{a_i}}(z)=a_i^{-1}$ and with dilation under $\boxtimes$. If $a_i=0$, the logarithmic Laplace limit is identically zero near $0$, while the $S$-transform is not defined.
\end{remark}

\begin{corollary}
\label{cor:local-free-deterministic}
Fix $k\ge2$. Let $X_{1,N},\ldots,X_{k,N}\in M_N(\C)$ be independent random matrices with deterministic singular values, and let $x_{1,N},\ldots,x_{k,N}$ be independent Haar unit vectors, independent of the matrices.
Assume that there is $0<M<\infty$ such that
$\norm{X_{i,N}}\le M$
for every $i,N$, and that
\begin{align*}
  \mu_{X_{i,N}^*X_{i,N}}\longrightarrow\mu_i
\end{align*}
weakly for each $i$, where $\mu_1$ is compactly supported with non-zero first moment and $\mu_2,\ldots,\mu_k$ are compactly supported non-degenerate probability measures on $[0,\infty)$. Then, for every real $t$ with
$|t|<\frac1{64M^2}$, the one-leg formula~\eqref{eq:one-leg-log-spherical-assumption} holds.
Consequently, after possibly shrinking the neighborhood of $0$ required in the iterative saddle-point argument, the conclusion of Theorem~\ref{thm:local-free} follows.
\end{corollary}

\begin{proof}
Set $H_{i,N}=X_{i,N}^*X_{i,N}$. Then $0\le H_{i,N}\le M^2\Id_N$, and $\mu_{H_{i,N}}$ is deterministic and converges weakly to $\mu_i$. For every $\delta>0$, the probability $\PP(d_{\mathrm L}(\mu_{H_{i,N}},\mu_i)>\delta)$ is therefore zero for all sufficiently large $N$, so the concentration hypothesis of Proposition~\ref{prop:loewner-squeeze} holds. Applying that proposition with $C=M^2$ gives~\eqref{eq:one-leg-log-spherical-assumption} for every real $|t|<1/(64M^2)$. The remaining hypotheses of Theorem~\ref{thm:local-free} hold by assumption, so its conclusion follows on a possibly smaller neighborhood of $0$.
\end{proof}

\begin{remark}
Corollary~\ref{cor:local-free-deterministic} does not require convergence of the extreme eigenvalues.
By contrast, Theorem~\ref{thm:rank-one-hciz} gives the one-leg formula on the larger, measure-dependent domain $\mathcal D_{\mu_i}$ under convergence of the extreme eigenvalues to the corresponding support edges; see \cite{GuionnetMaida2005}.
\end{remark}

\begin{remark}
Proposition~\ref{prop:loewner-squeeze} also applies when the singular values are random, provided that the uniform norm bound is retained and, for each $i$, $\mu_{X_{i,N}^*X_{i,N}}$ concentrates superexponentially at speed $N$ around $\mu_i$ as in~\eqref{eq:loewner-squeeze-spectral-concentration}.
\end{remark}

\begin{example}[Marchenko–Pastur and Fuss–Catalan laws]
We apply Corollary~\ref{cor:local-free-deterministic} to Marchenko--Pastur limiting Gram measures; the resulting limit is expressed through the $R$-transform of a Fuss--Catalan law.
Let $\pi$ be the Marchenko--Pastur law of parameter one, supported on $[0,4]$. For each $i=1,\ldots,k$, choose deterministic diagonal matrices $D_{i,N}$ such that the empirical spectral measures of $D_{i,N}^*D_{i,N}$ converge weakly to $\pi$ and such that $\norm{D_{i,N}}\leq2$.
Let $U_{i,N},V_{i,N}$ be independent Haar unitary matrices, independent across $i$, and set
$X_{i,N}=U_{i,N}D_{i,N}V_{i,N}^*$.
Corollary~\ref{cor:local-free-deterministic} applies with $M=2$.
Since $S_\pi(z)=\frac1{1+z}$, the free-convolution measure in Theorem~\ref{thm:local-free} is $\eta=\pi^{\boxtimes k}$ and $S_\eta(z)=(1+z)^{-k}$.
Its moments are the Fuss--Catalan numbers
\begin{align*}
  m_n(\eta)=\frac1{kn+1}\binom{(k+1)n}{n},
  \qquad n\geq0;
\end{align*}
see~\cite{BanicaBelinschiCapitaineCollins2011,PensonZyczkowski2011}.
Thus, for the product-state variable $Z_N$ in \eqref{eq:local-product-state-variable},
\begin{align*}
  \lim_{N\to\infty}\frac1N\log\E e^{NtZ_N}
  =\int_0^tR_{\pi^{\boxtimes k}}(s)\,\dd s,
\end{align*}
for $t$ in a neighborhood of zero.

The limiting Gram measures alone do not justify replacing this bounded model by Ginibre matrices.
For independent normalized complex Ginibre matrices $G_{1,N},\ldots,G_{k,N}$ with entry variance $1/N$, the empirical Gram measures converge weakly to the Marchenko--Pastur law. However, for $k\ge2$ and $t>0$, the product-state exponential moment is infinite at every finite $N$: the variables $Q_i=\norm{G_{i,N}x_{i,N}}^2$ are independent Gamma variables with shape and rate $N$, so $\E e^{NtqQ_k}=\infty$ whenever $tq\ge1$. Since $\prod_{i=1}^{k-1}Q_i\ge1/t$ has positive probability, conditioning gives $\E\exp\{Nt\prod_{i=1}^kQ_i\}=\infty$.
\end{example}

\section{Three or more tensor factors}
\label{sec:higher-leg-obstruction}

The exact identity of Theorem~\ref{thm:exact-two-leg-hciz} might suggest that a similar relation between the tensor product and the ordinary matrix product persists for a larger fixed number of factors.
We identify the annealed limit for the ordinary matrix product and show that it is strictly larger than the tensor-product limit for all sufficiently small fixed $t>0$, unless the limiting measures satisfy the equality conditions in Corollary~\ref{prop:higher-leg-counterexample}.

Fix $k\ge3$,
and let $\xi_N$ be Haar distributed on the unit sphere of $(\C^N)^{\otimes k}$, independently of the matrix factors. Define
\begin{align*}
  \mathcal X_N=X_{1,N}\otimes\cdots\otimes X_{k,N},
  \qquad
  \mathcal P_N=(X_{1,N}\cdots X_{k,N}) \otimes\Id_N^{\otimes(k-1)}
\end{align*}
and the annealed free energies
\begin{equation}
  \begin{aligned}
    F_N^\otimes(t)
    &=\frac1{N^k}\log\E_{X_1,\ldots,X_k,\xi_N}
    \exp\{N^k t\norm{\mathcal X_N\xi_N}^2\},\\
    F_N^{\mathrm{prod}}(t)
    &=\frac1{N^k}\log\E_{X_1,\ldots,X_k,\xi_N}
    \exp\{N^k t\norm{\mathcal P_N\xi_N}^2\}.
  \end{aligned}
\end{equation}

\begin{proposition}
\label{prop:higher-leg-aligned-spectrum}
Fix $k\ge3$. Let $X_{1,N},\ldots,X_{k,N}\in M_N(\C)$ be independent bi-unitarily invariant matrices with deterministic singular values and $\norm{X_{i,N}}\le M$ for all $i,N$.
Write $\lambda_{i,1}\ge\cdots\ge\lambda_{i,N}\ge0$ for the eigenvalues of $X_{i,N}^*X_{i,N}$, and set $c_j=\prod_{i=1}^k\lambda_{i,j}$ and $C_N=\diag(c_1,\ldots,c_N)$.
Then, for every fixed $t\ge0$,
\begin{align*}
F_N^{\mathrm{prod}}(t)-\frac1{N^k}\log I_{N^k}\left(t;C_N\otimes\Id_N^{\otimes(k-1)}\right)\longrightarrow0.
\end{align*}
If, in addition, $\mu_{X_{i,N}^*X_{i,N}}\to\mu_i$ weakly for each $i$, define
\begin{align*}
\nu_*=\Law\left(\prod_{i=1}^kQ_{\mu_i}(V)\right),
\end{align*}
where $V$ is uniform on $(0,1)$ and $Q_\mu(u)=\inf\{a\in\R:\mu((-\infty,a])\ge u\}$.
Then, for every sufficiently small fixed $t\ge0$,
\begin{align*}
\lim_{N\to\infty}F_N^{\mathrm{prod}}(t)=\int_0^tR_{\nu_*}(s)\dd s.
\end{align*}
\end{proposition}

\begin{proof}
Set $S_{i,N}=\diag(\sqrt{\lambda_{i,1}},\ldots,\sqrt{\lambda_{i,N}})$.
By bi-unitary invariance and independence, the singular values of $X_{1,N}\cdots X_{k,N}$ have the same joint law as those of
\begin{align*}
T_N=S_{1,N}W_{1,N}S_{2,N}\cdots W_{k-1,N}S_{k,N},
\end{align*}
where the $W_{i,N}$ are independent Haar unitaries.

Let $h_1\ge\cdots\ge h_N$ be the eigenvalues of $T_N^*T_N$.
Applying \cite[Theorem~3]{Horn1950} with $f(x)=x^2$ successively to the factors, and using summation by parts, gives $\sum_{j=1}^{\ell}h_j\le\sum_{j=1}^{\ell}c_j$ for $1\le\ell\le N$.
For $t\ge0$, $I_{N^k}\left(t;\diag(a)\otimes\Id_N^{\otimes(k-1)}\right)$ is a symmetric convex function of $a$ that is nondecreasing in each coordinate. The preceding partial-sum inequalities therefore imply
\begin{align*}
I_{N^k}\left(t;T_N^*T_N\otimes\Id_N^{\otimes(k-1)}\right)\le I_{N^k}\left(t;C_N\otimes\Id_N^{\otimes(k-1)}\right).
\end{align*}

Fix $\delta\in(0,1)$.
The covering estimate in \cite[Theorem~7]{Szarek1998MetricEntropy}, together with Haar invariance and independence, gives
\begin{align*}
\PP\left(\max_{1\le i<k}\norm{W_{i,N}-\Id_N}\le\delta\right)\ge e^{-c_\delta N^2},
\end{align*}
where $c_\delta$ is independent of $N$.
On this event, replacing the $W_{i,N}$ by $\Id_N$ one at a time gives $\norm{T_N-C_N^{1/2}}\le(k-1)M^k\delta$, and hence $\norm{T_N^*T_N-C_N}\le C\delta$, where $C$ depends only on $k$ and $M$.
Consequently,
\begin{align*}
I_{N^k}\left(t;T_N^*T_N\otimes\Id_N^{\otimes(k-1)}\right)\ge e^{-N^ktC\delta}I_{N^k}\left(t;C_N\otimes\Id_N^{\otimes(k-1)}\right).
\end{align*}
Restricting the expectation over the $W_{i,N}$ to this event and using its probability bound gives
\begin{align*}
0\le\frac1{N^k}\log I_{N^k}\left(t;C_N\otimes\Id_N^{\otimes(k-1)}\right)-F_N^{\mathrm{prod}}(t)\le Ct\delta+c_\delta N^{2-k}.
\end{align*}
Since $N^{2-k}\to0$ for $k\ge3$, letting $N\to\infty$ and then $\delta\downarrow0$ proves the first assertion.

If $\mu_{X_{i,N}^*X_{i,N}}\to\mu_i$ weakly for each $i$, then $Q_{\mu_{X_{i,N}^*X_{i,N}}}(u)\to Q_{\mu_i}(u)$ for almost every $u\in(0,1)$.
Since $\mu_{C_N}$ is the law of $\prod_{i=1}^kQ_{\mu_{X_{i,N}^*X_{i,N}}}(V)$, it follows that $\mu_{C_N}\to\nu_*$ weakly.
The matrix $C_N\otimes\Id_N^{\otimes(k-1)}$ has the same empirical spectral measure as $C_N$ and norm at most $M^{2k}$.
The second assertion therefore follows from \cite[Theorem~2]{GuionnetMaida2005}, with $\beta=2$, for sufficiently small $t\ge0$.
\end{proof}

\begin{corollary}
\label{prop:higher-leg-counterexample}
Fix $k\ge3$ and assume the hypotheses of Proposition~\ref{prop:higher-leg-aligned-spectrum} and that $\mu_{X_{i,N}^*X_{i,N}}\to\mu_i$ weakly for every $i$.
For every sufficiently small fixed $t>0$, both $F_N^{\mathrm{prod}}(t)$ and $F_N^\otimes(t)$ converge.
If $\mu_i\ne\delta_0$ for every $i$ and at least two of the $\mu_i$ are not point masses, then
\begin{align}\label{eq:higher-leg-strict-gap}
\lim_{N\to\infty}F_N^{\mathrm{prod}}(t)>\lim_{N\to\infty}F_N^\otimes(t).
\end{align}
The two limits are equal if and only if some $\mu_i=\delta_0$ or at least $k-1$ of the $\mu_i$ are point masses.
\end{corollary}

\begin{proof}
Set $\nu_{\mathrm{cl}}=\mu_1\circledast\cdots\circledast\mu_k$.
Proposition~\ref{prop:higher-leg-aligned-spectrum} and Theorem~\ref{thm:global-classical} give the limits $\int_0^tR_{\nu_*}(s)\dd s$ and $\int_0^tR_{\nu_{\mathrm{cl}}}(s)\dd s$, respectively.
If some $\mu_i=\delta_0$, then $\nu_*=\nu_{\mathrm{cl}}=\delta_0$.
If at least $k-1$ of the $\mu_i$ are point masses, all but possibly one factor are constant, so again $\nu_*=\nu_{\mathrm{cl}}$.

Outside these two cases, let $V'$ be an independent copy of $V$.
For bounded nondecreasing functions $f,g$ on $(0,1)$,
\begin{align*}
\E[f(V)g(V)]-\E f(V)\E g(V)=\frac12\E[(f(V)-f(V'))(g(V)-g(V'))]\ge0.
\end{align*}
If neither function is almost everywhere constant, the inequality is strict, since both differences are positive when $V$ is sufficiently close to $1$ and $V'$ sufficiently close to $0$.
Each $Q_{\mu_i}$ is nonnegative and nondecreasing, and $\E Q_{\mu_i}(V)=m_1(\mu_i)>0$.
Applying the preceding inequality successively, starting with two nonconstant functions $Q_{\mu_i}$, gives $m_1(\nu_*)>\prod_{i=1}^km_1(\mu_i)=m_1(\nu_{\mathrm{cl}})$.
Since $R_\rho(s)=m_1(\rho)+O(s)$ near $0$ for a compactly supported probability measure $\rho$,
\begin{align*}
\lim_{N\to\infty}\left(F_N^{\mathrm{prod}}(t)-F_N^\otimes(t)\right)=\left(m_1(\nu_*)-m_1(\nu_{\mathrm{cl}})\right)t+O(t^2)
\end{align*}
as $t\downarrow0$, proving~\eqref{eq:higher-leg-strict-gap}.
\end{proof}

\begin{example}
Let $X_{1,N}$ and $X_{2,N}$ be independent bi-unitarily invariant matrices, each having $\lfloor N/2\rfloor$ squared singular values equal to $2$ and the rest equal to $1$.
Let $X_{3,N},\ldots,X_{k,N}$ be independent Haar unitaries, independent of $X_{1,N},X_{2,N}$.
Then $\mu_1=\mu_2=\frac12\delta_1+\frac12\delta_2$ and $\mu_3=\cdots=\mu_k=\delta_1$, so
\begin{align*}
\nu_*=\frac12\delta_1+\frac12\delta_4,\qquad \mu_1\circledast\cdots\circledast\mu_k=\frac14\delta_1+\frac12\delta_2+\frac14\delta_4.
\end{align*}
Proposition~\ref{prop:higher-leg-aligned-spectrum} and Theorem~\ref{thm:global-classical} therefore give
\begin{align*}
\lim_{N\to\infty}F_N^{\mathrm{prod}}(t)&=\frac52t+O(t^2),\\
\lim_{N\to\infty}F_N^\otimes(t)&=\frac94t+O(t^2)
\end{align*}
as $t\downarrow0$.
\end{example}

\section*{Acknowledgments}

B.C. and M.N. were supported by JSPS Grant-in-Aid for Scientific Research (A), no.~25H00593, and Challenging Research (Exploratory), no.~23K17299. B.C. acknowledges support from the Ministry of
Education, Culture, Sports, Science and Technology (MEXT) through SPReaD (Supporting Pioneering Research through AI for 1,000 Discovery challenges). M.N. acknowledges support from JST SPRING, Grant Number
JPMJSP2110.
B.C., N.D., and R.T. acknowledge the support from
Okinawa Institute of Science and Technology (OIST) SHINKA Grant.
The authors thank the thematic program ``Random tensors and related topics" at Institut Henri Poincar\'e (IHP), Paris, France from September 30 - October 18, 2024, where some of the work in this paper has been carried out.
The authors also acknowledge stimulating discussions with Alice Guionnet, Razvan Gurau, and Luca Lionni. B.C. is grateful to
R.T. and OIST for hospitality at the beginning of this work.

Generative AI tools were used towards the completion of this work and helped tighten some mathematical arguments, provide relevant examples, and edit the text. All mathematical statements, arguments, and references were independently checked and edited by the authors, who take full responsibility for the contents of the paper.

\bibliographystyle{alpha}
\bibliography{references}

\end{document}